\documentclass[11pt]{article}
\usepackage{graphicx}
\usepackage{amsmath}
\usepackage{amssymb}
\usepackage{amsthm}
\usepackage{epsfig}
\usepackage{mathrsfs}
\usepackage{multirow}
\usepackage[usenames,dvipsnames]{color}
\usepackage{setspace}
\usepackage{algorithm, algorithmic} 
\usepackage{multicol}
\usepackage{subcaption}
\usepackage[margin=2.54cm]{geometry}
\usepackage{lipsum}
\usepackage{bigints}
\usepackage{framed}
\usepackage{colortbl}
\usepackage{calligra}
\usepackage{mathrsfs}
\usepackage{tabularx,hhline}

\DeclareMathOperator{\st}{s.t.}
\DeclareMathOperator{\diag}{diag}
\DeclareMathOperator{\Diag}{Diag}

\DeclareMathAlphabet{\mathcalligra}{T1}{calligra}{m}{n}

\newtheorem{theorem}{Theorem}

\newtheorem{lemma}{Lemma}

\newtheorem{proposition}{Proposition}
\newtheorem{assumption}{Assumption \!\!}

\newtheorem{remark}{Remark}
\newtheorem{example}{Example}

\newcommand{\K}{{\mathcal K}}
\newcommand{\U}{{\mathcal U}}

\newcommand{\Uzero}{{\mathcal U^0}}
\newcommand{\Ubar}{{\mathcal U}'}

\newcommand{\CP}{\mathcal{CP}}
\newcommand{\COP}{\mathcal{COP}}
\newcommand{\SOC}{\mathcal{SOC}}
\newcommand{\IA}{{\mathcal{IA}}}
\newcommand{\AS}{{\mathcal{AS}}}

\newcommand{\SYM}{{\mathcal S}}

\newcommand{\bm}{\boldsymbol}

\def\RR{ {\Bbb{R}}}

\newcommand{\X}{{\mathcal X}}
\newcommand{\A}{{\bm{\mathcal A}}}
\newcommand{\BB}{{\bm{\mathcal B}}}

\newcommand{\uu}{{\bm{u}}}
\newcommand{\vv}{{\bm{v}}}
\newcommand{\dd}{{\bm{d}}}

\newcommand{\yy}{{\mathscr{\bm{y}}}}

\newcommand{\x}{{\bm{x}}}
\newcommand{\cc}{{\bm{c}}}

\newcommand{\hh}{{\bm{h}}}

\newcommand{\FF}{{\mathscr{F}}}

\newcommand{\LDR}{{\mathcal{L}} }
\newcommand{\QDR}{{\mathcal{Q}} }
\newcommand{\PLDR}{{\mathcal{PL}} }
\newcommand{\PQDR}{{\mathcal{PQ}} }

\title{ Improved Decision Rule Approximations for Multi-Stage Robust Optimization via Copositive Programming }

\author{
Guanglin Xu\thanks{Institute for Mathematics and its Applications, University of Minnesota,
Minneapolis, MN, 55455, USA. Email: {\tt gxu@umn.edu}.}%
\ \ \ \ \ \ \
Grani A. Hanasusanto\thanks{Graduate Program in Operations Research and Industrial Engineering, 
The University of Texas at Austin, Austin, TX, 78712-1591, USA. Email: {\tt grani.hanasusanto@utexas.edu}.}%
}

\date{ }

\begin{document}

\maketitle

\begin{abstract}

\noindent We study decision rule approximations for generic multi-stage 
robust linear optimization problems. We examine linear decision rules for the 
case when the objective coefficients, the recourse matrices, and the right-hand 
sides are uncertain, and examine quadratic decision rules for the case when only
the right-hand sides are uncertain. The resulting optimization problems are 
NP-hard but amenable to copositive programming reformulations that give rise 
to tight, tractable semidefinite programming solution approaches. We further enhance these approximations through new piecewise decision rule schemes. Finally, we prove that our proposed
approximations are tighter than the state-of-the-art schemes and demonstrate their superiority through numerical experiments. 

\mbox{}


\end{abstract}



\begin{onehalfspace}

\section{Introduction} \label{sec:intro}

Decision-making under uncertainty arises in a wide spectrum of  applications 
in operations management, engineering, finance, and process control. 
A prominent modeling approach 
for decision-making under uncertainty is {\em robust optimization} (RO), whereby
one seeks for a decision that hedges against the worst-case realization of uncertain 
parameters; see~\cite{Ben-Tal.Ghaoui.Nemirovski.2009, Ben-Tal.Nemirovski.2002, 
Bertsimas.Brown.Caramanis.2011}. RO paradigm is appealing because it leads to 
computationally tractable solution schemes for many \emph{static} decision-making 
problems under uncertainty. However, real-life  problems are often \emph{dynamic} 
in nature, where the uncertain parameters are revealed sequentially and the decisions 
must be adapted to the current realizations. The adaptive decisions are fundamentally 
infinite-dimensional as they constitute mappings from the space of uncertain 
parameters to the space of actions. This setting gives rise to the {\em multi-stage 
robust optimization} (MSRO) problems which in general are  computationally challenging to solve. 
Only in a few cases and under very stringent conditions are the problems efficiently solvable;  
see for instance~\cite{Ben-Tal.Nemirovski.1999,Bertsimas.Goyal.2015,Guslitser.2002}. 
Consequently, the design of solution schemes for MSRO  necessitates to reconcile the 
conflicting objectives of optimality and scalability. 

Conservative approximations for MSRO can be derived in \emph{linear decision rules}, 
where we restrict the adaptive decisions to affine functions in the uncertain parameters. 
Popularized by Ben-Tal et al.~\cite{Ben-Tal.Goryashko.Guslitzer.Nemirovski.2004}, 
linear decision rules have found successful applications in various areas of decision-making
problems under uncertainty~\cite{Atamturk.Zhang2007,Ben-Tal.Golany.Nemirovski.Vial.2005,
Calafiore2008,Chen.Sim.Sun.2007,Chen.Sim.Peng.Zhang.2008,Gounaris.Wiesemann.Floudas.2013,
Rocha.Kuhn.2012} as they are simple yet reasonable to implement in practice. Moreover, linear decision 
rules are optimal for some instances of  MSRO~\cite{Bertsimas.Goyal.2012,Iancu.Sharma.Sviridenko.2013}, 
linear quadratic optimal control~\cite{Anderson.Mooore.2007}, and robust vehicle
routing~\cite{Gounaris.Wiesemann.Floudas.2013} problems. The resulting optimization problems, 
however, are tractable only under the restrictive setting of  \emph{fixed recourse}, i.e., 
when the adaptive decisions are not multiplied with the uncertain parameters in the problem's formulation. 
Many decision-making problems under uncertainty such as portfolio 
optimization~\cite{Ben-Tal.Margalit.Nemirovski.2000,Dantzig.Infanger.1993,Rocha.Kuhn.2012}, 
energy systems operation planning~\cite{Georghiou.Wiesemann.Kuhn.2015-2,Martins.2013}, 
inventory planning~\cite{Bertsimas.Georghiou.2018}, etc.~do not satisfy the fixed recourse assumption. 
For these problem instances, the linear decision rule approximation is NP-hard already in a two-stage 
setting~\cite{Ben-Tal.Goryashko.Guslitzer.Nemirovski.2004,Guslitser.2002}.

The basic linear decision rules have been extended to truncated linear~\cite{See.Sim.2010}, 
segregated linear~\cite{Chen.Sim.Peng.Zhang.2008, Chen.Zhang.2009, Goh.Sim.2010}, and 
piecewise linear~\cite{Ben-Tal.El-Housni.Goyal.2018, Georghiou.Wiesemann.Kuhn.2015}
functions in the uncertain parameters. If the MSRO problem has fixed recourse then one can 
formally prove that the optimal adaptive decisions are piecewise linear~\cite{Bemporad.Borrelli.Morari.2003}, 
which justifies the use of these enhanced approximations. Unfortunately, optimizing for the best piecewise 
linear decision rule entails solving globally a non-convex optimization problem which is inherently 
difficult~\cite{Ben-Tal.El-Housni.Goyal.2018,Bertsimas.Georghiou.2015}. If in addition some basic 
descriptions about the piecewise linear structure are prescribed, then one can derive tractable linear 
programming approximations for problem instances with fixed recourse by Georghiou et al.~\cite{Georghiou.Wiesemann.Kuhn.2015}. 
Their piecewise linear decision rule scheme is a generalization of the aforementioned methods, including the
 truncated linear decision rule~\cite{See.Sim.2010} and the segregated linear decision 
rule~\cite{Chen.Sim.Peng.Zhang.2008, Chen.Zhang.2009, Goh.Sim.2010}.

If tighter approximation is desired or when the  problem has non-fixed recourse, then one can  in principle  
develop a hierarchy of increasingly tight semidefinite  approximations using 
\emph{polynomial decision rules}~\cite{Bertsimas.Iancu.Parrilo.2011}.  While optimizing for the best 
polynomial decision rule of fixed degree is difficult, tractable conservative approximations can be obtained 
by employing the Lasserre hierarchy~\cite{Lasserre.2009,Parrilo.2000}. Such approximations are attractive 
because they do not require prior structural knowledge about the optimal adaptive decisions. However, 
the resulting semidefinite programs scale poorly with the degree of the polynomial decision rules. A decent tradeoff 
between suboptimality and scalability is attained in \emph{quadratic decision rules}, where one merely optimizes 
over polynomial functions of degree~$2$. Their semidefinite approximations, based on the well-known approximate 
S-lemma~\cite{Ben-Tal.Ghaoui.Nemirovski.2009},  have been applied successfully to instances of inventory 
planning~\cite{Bertsimas.Iancu.Parrilo.2011,Hanasusanto.Kuhn.Wallace.Zymler.2015} and
electricity capacity expansion~\cite{Bampou.Kuhn.2011} problems. A posteriori lower bounds to the MSRO 
problem can be derived by applying  decision rules to  the problem's dual formulation~\cite{Bampou.Kuhn.2011,
Georghiou.Wiesemann.Kuhn.2015, Kuhn.Wiesemann.Georghiou.2011}.  Alternative schemes that  similarly 
provide aggressive bounds for MSRO are proposed in \cite{Hadjiyiannis.Goulart.Kuhn.2011} 
and~\cite{Bertsimas.deRuiter.2016}. All the methods mentioned above can be applied to different paradigms in optimization under uncertainty, such as stochastic programming, robust optimization, and distributionally robust optimization. 
Our paper focuses on the robust optimization setting because it requires minimal assumptions about the uncertainty, which allows us to present the main idea cleanly. If distributional information is available, then the proposed method can be directly applied to the other settings in a relatively straightforward fashion.

Global optimization approaches have also been designed to derive exact solutions of 
MSRO problems. In the two-stage robust optimization setting, these methods include Benders' 
decomposition~\cite{Bertsimas.Litvinov.Sun.Zhao.Zheng.2013, 
	Doulabi.Jaillet.Pesant.Rousseau.2016}, column and constraint 
generation~\cite{Zeng.Zhao.2013}, extreme point enumeration combined with decision 
rules \cite{Georghiou.Tsoukalas.Wiesemann.2017}, and Fourier-Motzkin elimination~\cite{Zhen.Hertog.Sim.2016}.  
The Benders' decomposition scheme has been extended to the multi-stage setting for MSRO 
problems where the uncertain parameters exhibit a \emph{stagewise rectangular} 
structure~\cite{Georghiou.Tsoukalas.Wiesemann.2016}. The papers   \cite{Bertsimas.Dunning.2016} 
and \cite{Postek.Hertog.2016} develop adaptive uncertainty set partitioning schemes that generate a 
sequence of increasingly accurate conservative approximations for MSRO. Global optimization scheme 
has also been conceived through the lens of  conic reformulations. 
Hanasusanto and Kuhn~\cite{Hanasusanto.Kuhn.2016}
and Xu and Burer~\cite{Xu.Burer.2018} propose independently 
equivalent copositive programming reformulations for two-stage robust 
optimization problems and develop conservative semidefinite 
approximations for the reformulations. 
 
Using {\em copositive programming} techniques, this paper takes a first step towards addressing a generic linear MSRO problem 
where the objective coefficients, the recourse matrix, and the right-hand sides are uncertain. A copositive program is a convex program that optimizes a linear function over  the cone of copositive matrices subject to linear constraints~\cite{Bomze.2012, Burer.2012, Dur.2010}. Bomze et al.~\cite{Bomze.Dur.Klerk.Roos.2000} are the first to reformulate an NP-hard problem, namely the standard quadratic problem, to an 
equivalent copositive program.  The seminal work of Burer~\cite{Burer.2009} shows that a generic quadratic program can be reformulated to an equivalent copositive program. 
In another work, Burer and Dong~\cite{Burer.Dong.2012} establish the equivalence between a non-convex quadratically constrained quadratic program (QCQP) and a generalized copositive program under certain conditions.  We refer the reader to~\cite{Burer.Dong.2012, Chen.Burer.2012, Kong.Lee.Teo.Zheng.2013, Natarajan.Teo.2017, Natarajan.Teo.Zheng.2011} for more works on using copositive techniques to reformulate non-convex quadratic programs arising in different applications.

Our {\em key contribution} is to utilize copositive 
programming techniques to develop {\em stronger} decision rule approximations for {\em generic} MSRO problems. 
In the generic settings, the direct use of decision rules leads to computationally intractable semi-infinite programs, with finitely many decision variables but infinitely many constraints. The standard dualization procedure in robust optimization does not apply because these constraints involve non-convex QCQPs. We leverage the copositive reformulation techniques to convexify the QCQPs, which enables the dualization of the constraints to arrive at finite-dimensional convex optimization problems. 
The copositive techniques further allow us to handle complex uncertainty sets (e.g., integrating complementary constraints), which lead to exact convex reformulations for a class of piecewise decision rule approximations. All these new reformulations enjoy tractable semidefinite approximations that are provably superior to the state-of-the-art schemes.
We summarize the contributions of the paper as follows.
\begin{enumerate}
	\item For the generic MSRO problems we derive new copositive programming 
	reformulations in view of the popular linear decision rules. For MSRO problems 
	with  fixed recourse we derive new  copositive programming reformulations in view of the 
	more powerful quadratic decision rules. The exactness results are   general: They hold for 
	MSRO problems \emph{without} relatively complete recourse, and under  very minimal 
	assumption about the compactness of the uncertainty set, \emph{without} requiring it to 
	exhibit stage-wise rectangularity. 
	\item The emerging copositive programs are amenable to a hierarchy of  increasingly tight 
	conservative semidefinite programming approximations. We formulate the simplest of these 
	approximations and prove that it is tighter than the state-of-the-art scheme by 
	Ben-Tal et al.~\cite{Ben-Tal.Goryashko.Guslitzer.Nemirovski.2004}, and also  the polynomial 
	decision rule scheme by Bertsimas et al.~\cite{Bertsimas.Iancu.Parrilo.2011} when the degree 
	of the polynomial is set to the degree of our decision rules (degree $1$ for problems with 
	non-fixed recourse and degree $2$ for problems with fixed recourse). We demonstrate 
	empirically  that our proposed  approximation is competitive to  polynomial decision rules 
	of higher degrees while displaying more favorable scalability. 
	\item We propose piecewise linear decision rules for MSRO problems with non-fixed recourse 
	and piecewise quadratic decision rules for MSRO problems with fixed recourse. 
	To our best knowledge, these decision rules are new for their respective problem classes. 
	By leveraging  recent  techniques in copositive programming, we derive equivalent copositive 
	programs for the  piecewise decision rule approximations.   
	For MSRO problems with fixed recourse, we show that the state-of-the-art scheme by Georghiou 
	et al.~\cite{Georghiou.Wiesemann.Kuhn.2015} can be futile even on trivial two-stage problem 
	instances, while our semidefinite approximation produces high-quality solutions. We formally 
	prove that our proposed  approximation is indeed tighter than that 
	of~\cite{Georghiou.Wiesemann.Kuhn.2015}, and further identify the simplest set of semidefinite 
	constraints that retains the outperformance while maintaining scalability. 
\end{enumerate}


The remainder of the paper is organized as follows. We derive the copositive programming reformulations 
for two-stage robust optimization problems in Section~\ref{sec:two-stage}. 
In Section~\ref{sec:SDP_approx}, 
we develop the conservative semidefinite programming approximations. We 
extend all results to the multi-stage setting in Section \ref{sec:multi-stage} and present the numerical 
results in Section \ref{sec:numerical}.  

\subsection{Notation and terminology}

For any $M \in \Bbb{N}$, we define $[M]$ as the set of running indices
$\{1, \ldots, M\}$. We let $[M]\backslash \{1\}$ be the set of running indices $\{2,\ldots,M\}$.
We denote by  $\mathbf{e}$ the vector of all ones  and by $\mathbf{e}_i$ the $i$-th 
standard basis vector. For notational convenience, we use both $v_i$ and 
$[\bm v]_i$ to denote the $i$-th component of the vector $\bm{v}$.  
The $p$-norm of a vector $\bm{v} \in \RR^N$ is defined as~$\|\bm{v}\|_p$. 
We will drop the subscript for the Euclidean norm, i.e., $\|\bm{v}\| := \|\bm{v}\|_2$.  
For $\bm a \in \RR^N$ and $\bm b \in \RR^N$, the Hadamard product of 
$\bm a$ and $\bm b$ is denoted by $\bm a \circ \bm b:= (a_1b_1, \ldots, a_Nb_N)^\top$. 
The trace of a square matrix $\bm X$ is denoted as $\text{trace} (\bm {X})$.
We use $[\bm A]_{ij}$ to denote the entry in the $i$-th row and the $j$-th column of
the matrix $\bm{A}$. We define $\diag(\bm{X})$ as the vector comprising the diagonal 
entries of $\bm X$, and $\Diag(\bm{v})$ as the diagonal matrix with the vector $\bm{v}$ 
along its main diagonal. We use $\bm{X} \geq \bm 0$ to denote that $\bm{X}$ is a component-wise 
nonnegative matrix. For any matrix $\bm{A} \in \RR^{M \times N}$, the inclusion
$\text{Rows}(\bm{A}) \in \K$ indicates that the column vectors corresponding to the rows of $\bm{A}$ 
are members of $\K$. We denote by  
$\FF_{K+1,\, N}$ the space of all measurable  mappings $\yy(\cdot)$ from $\RR^{K+1}$ to $\RR^N$. 

For any closed and convex cone  ${\cal K}$, we denote its dual cone as ${\cal K}^*$. 
We define by ${\SOC} \subseteq \RR^{K+1}$  the standard second-order cone, i.e.,
$\bm{v} \in {\SOC} \Longleftrightarrow \|(v_1, \ldots, v_K)^\top\| \leq v_{K+1}$.  
We denote the space of symmetric matrices in $\mathbb R^{N\times N}$ as~$\SYM^N$. 
For any $\bm{X} \in \SYM^N$, 
we set $\bm{X} \succeq \bm 0$ to denote that $\bm{X}$ is positive semidefinite. 
For convenience, we call the cone of positive semidefinite matrices as the semidefinite cone 
and the cone of symmetric nonnegative matrices as the the nonnegative cone. 
The {\em copositive cone\/} is defined as
 $\COP(\RR_+^N) := \{ \bm{M} \in \SYM^N: \x^\top \bm{M} \x \ge 0 \ \forall \x \in \RR_+^N \}$.
Its dual cone, the {\em completely positive cone\/}, is defined as 
    $\CP(\RR_+^N) := \{ \bm{X} \in \SYM^N: \bm{X} = \textstyle{\sum_i} \x^i (\x^i)^\top, \ \x^i \in \RR_+^N \}$,
where the summation over $i$ is finite but its cardinality is
unspecified. For a general closed and convex cone
${\cal K} \subseteq \RR^N$, we define the {\em generalized
copositive cone} as $\COP({\cal K})$ and the 
{\em generalized completely positive cone\/} as
$\CP({\cal K})$, respectively, in analogy with $\COP(\RR_+^N)$ and
$\CP(\RR_+^N)$. Note that $\COP({\cal K})$
and $\CP({\cal K})$ are dual cones to each other. 
The term {\em copositive programming\/} refers to linear
optimization over $\COP({\cal K})$ or, via duality, linear optimization
over $\CP({\cal K})$.  To distinguish from 
 the standard case where ${\cal K} = \RR_+^N$, they are  sometimes called
{\em generalized copositive programming\/} or {\em set-semidefinite
optimization\/}~\cite{Burer.Dong.2012, Eichfelder.Jahn.2008}. In this paper, we work with
generalized copositive programming, although we use the shorter phrase
for simplicity.

\section{Copositive reformulations for two-stage decision rule problems} \label{sec:two-stage}

In this section, we first state the generic setting of a two-stage robust optimization problem. We then consider various decision rules for the two-stage problem and propose copositive programming reformulations for the decision rule problems.

\subsection{Two-stage robust optimization problem}

We study adaptive linear optimization problems of the following general structure. 
A decision maker first takes a here-and-now decision $\bm x \in\X$, which 
incurs an immediate linear cost $\bm c^\top \bm x$. 
Nature then reacts with  a worst-case parameter realization $\bm u\in\U$. 
In response, the decision maker takes a recourse action $\bm y(\bm u)\in\RR^N$, 
which incurs a second-stage linear cost $\dd(\uu)^{\top}\yy(\uu)$. 
 In this game against nature,  the decision maker endeavors  to optimally select a feasible solution 
$(\bm x,\yy(\cdot))$ that minimizes the total cost 
$\cc^{\top} \x + \sup_{\uu \in \U} \dd(\uu)^{\top}\yy(\uu)$. We note that the second-stage decision 
vector constitutes a mapping $\yy:  \U \rightarrow \RR^N$ and is thus infinite 
dimensional. 

The emerging sequential 
decision problem can be formulated as a two-stage robust optimization problem 
given by 
\begin{equation} \label{equ:tsro}
\begin{array}{rcl} 
Z=&\inf & \cc^{\top} \x + \sup \limits_{\uu \in \U} \dd(\uu)^{\top}\yy(\uu) \\
&\st  & \A(\uu)\x +  \BB(\uu) \yy(\uu) \geq  \hh(\uu)  \  \  \  \forall \, \uu \in \U  \\
&& \x \in \X, \ \yy \in \FF_{K+1, \, N}.
\end{array}
\end{equation}
Here, the feasible set of the first-stage decision $\x$ is captured by a generic set $\X \subseteq \RR^M$, 
while that of the second-stage decision $\bm y(\bm u)$ is defined through a  linear constraint system 
$\A(\uu)\x +  \BB(\uu) \yy(\uu) \geq  \hh(\uu)$. The uncertain parameter vector $\uu$ is assumed to 
belong to a prescribed \emph{uncertainty set}  $\U$, which we model as the intersection of 
a slice of a closed and convex cone $\K \subseteq \RR^K \times  \RR_+$, and the
level sets of $I$ quadratic functions. Specifically, we set 
\begin{equation} \label{equ:uncertainty}
\U := \left\{ \uu \in \K : \begin{array}{l}\mathbf{e}_{K+1}^\top \uu = 1\\[2mm]
  \uu^\top \bm{\widehat C}_i \uu = 0 \ \ \ \forall \, i \in [I]  \end{array}\right\},
\end{equation}
where $\bm{\widehat C}_i \in \SYM^{K+1}$ for all $i \in [I]$. The problem parameters $\A(\uu) \in 
\RR^{J \times M}$, $\BB(\uu) \in \RR^{J \times N}$,  $\dd(\uu) \in \RR^N$ and $\hh(\uu) \in \RR^J$
in \eqref{equ:tsro} are assumed to be linear in $\uu$, given by
\[
\begin{array}{l}
\A(\uu) = \sum \limits_{k=1}^{K+1} u_k \, \bm{\widehat A}_k, \ \ \  \BB(\uu) = 
\sum \limits_{k=1}^{K+1} u_k \, \bm{\widehat B}_k, \ \ \ 
\dd(\uu) =  \bm{\widehat D}\uu, \ \ \  
\hh(\uu) =   \bm{\widehat H}\uu,
\end{array}
\]
where $\bm{\widehat A}_k \in \RR^{J \times M}$, $\bm{\widehat B}_k \in \RR^{J \times N}$, 
$\bm{\widehat D} := (\bm{\widehat d}_1, \ldots, \bm{\widehat d}_{N})^\top \in \RR^{N \times (K+1)}$,
and $\bm{\widehat H} := (\bm{\widehat h}_1, \ldots, \bm{\widehat h}_{J})^\top \in \RR^{J \times (K+1)}$
are deterministic data. 
The nonrestrictive assumption that $u_{K+1}= 1$ in (\ref{equ:uncertainty}) 
will simplify notation as it allows us to represent affine functions in the 
primitive uncertain parameters $(u_1, \ldots, u_K)^{\top}$ in a 
compact way as linear functions of $\bm{u}$, e.g., the problem parameters 
$\A(\uu)$, $\BB(\uu)$, $\dd(\uu)$, and $\hh(\uu)$, and the linear decision rule
$\bm Y\uu$ (Section~\ref{sec:ldr}); and as it also allows us to represent 
quadratic functions in the primitive uncertain parameters 
in a homogenized manner, e.g., the quadratic decision rule $\uu^\top\bm Q\uu$ 
(Section~\ref{sec:ldr}). 

The cone $\K$ in the description of $\U$
has a   generic form and can  model many common uncertainty sets in the literature. 
We highlight three pertinent examples as follows. 

\begin{example}[Polytope]
	\label{exa:polytope}
	If the  uncertainty set of the primitive   vector $(u_1,\dots,u_K)^\top$ is given by a polytope 
	$\{\bm \xi \in \RR^{K}:\bm P\bm\xi\geq \bm q \}$, then the corresponding cone is defined as
	\[
	\K:= \left\{ (\bm \xi, \tau ) \in \RR^{K} \times \RR_+ : 
	\bm{P}\bm \xi \geq   \bm{q} \tau
	\right\}.
	\]
\end{example}

\begin{example}[Polytope and 2-Norm Ball]
	If the  uncertainty set of the primitive   vector is given by the intersection of  a polytope  and a transformed $2$-norm ball:
	$\{\bm \xi \in \RR^{K}:\bm P\bm\xi\geq \bm q,\;\|\bm R\bm\xi-\bm s\|_2\leq t\}$, then the corresponding cone is defined as
	\[
\K:= \left\{ (\bm \xi, \tau ) \in \RR^{K} \times \RR_+ : 
\bm{P}\bm \xi\geq   \bm{q} \tau,\,\|\bm R\bm \xi-\bm s\tau\| \leq t\tau
\right\}.
\]
\end{example}

\begin{example}[Ellipsoids]
	\label{exa:ellipsoids}
	Consider the setting where the uncertainty set of the primitive  vector is described by an intersection of $L$ ellipsoids:
	$\{ \bm \xi \in \RR^{K} : \bm \xi^\top \bm F_\ell \bm \xi + 2\bm g_\ell^\top \bm \xi \leq  h_\ell, \;\forall \ell \in [L]\}$.
	Here, $\bm F_\ell \in \SYM^K$,  $\bm F_\ell \succeq \bm 0$,  $\bm g_\ell \in \RR^{K}$,  and   $h_\ell \in \RR$ for all
	$\ell \in [L]$. Since $\bm F_\ell$ is positive semidefinite, we have 
	$\bm F_\ell = \bm P_\ell^\top \bm P_\ell$ for some matrix $\bm P_\ell \in \RR^{I_\ell \times K}$ whose rank is $I_\ell$. 
	In~\cite{Alizadeh.Goldfarb.2003}, it is shown that
	\[
	\bm \xi^\top \bm F_\ell \bm \xi + 2\bm g_\ell^\top \bm \xi  \leq  h_\ell \ \  \Longleftrightarrow  \ \
	\begin{pmatrix} 
	\bm P_\ell \bm \xi \\
	\frac12(1- h_\ell) + \bm g_\ell^\top \bm \xi  \\ 
	\frac12(1+ h_\ell) - \bm g_\ell^\top \bm \xi 
	\end{pmatrix} 
	\in \SOC(I_\ell + 2),
	\] 
	where $\SOC(I_\ell + 2)$ denotes the second-order cone of dimension $I_\ell + 2$.
	In this case, the corresponding cone is given by
	\[
	\K := \left\{ 
	(\bm \xi, \tau) \in   \RR^K  \times \RR_+ : 
	\begin{pmatrix} 
	\bm P_\ell \bm \xi \\
	\frac12(1- h_\ell)\tau + \bm g_\ell^\top \bm \xi  \\ 
	\frac12(1+ h_\ell)\tau - \bm g_\ell^\top \bm \xi 
	\end{pmatrix} 
	\in \SOC(I_\ell + 2)  \  \ \forall \, \ell \in [L] 
	\right\}.
	\] 
\end{example}

In the following, to simplify our exposition,  we define  the convex set
\begin{equation} \label{equ:Lset}
{\Uzero} := \left\{ \uu \in \K : \mathbf{e}_{K+1}^\top \uu = 1 \right\},
\end{equation}
which corresponds to the uncertainty set $\mathcal U$ in the absence 
of the non-convex constraints $\uu^\top \bm{\widehat C}_i \uu = 0$, $i \in [I]$. 
We further assume that the uncertainty set satisfies the following regularity conditions. 
 
\begin{assumption} \label{ass:bounded}
The set ${\Uzero} $ defined in (\ref{equ:Lset})  
is nonempty and compact. 
\end{assumption}

\begin{assumption} \label{ass:quadcons}
The minimum value of the quadratic function $\uu^\top \bm{\widehat C}_i\uu$ 
over the set ${\Uzero} $ is 0 for all $i \in [I]$, i.e., $
0 = \min_{\uu \, \in \, {\Uzero} } \uu^\top \bm{\widehat C}_i\uu$, $i \in [I]$. 
\end{assumption}

\noindent The quadratic constraints in the description of 
$\U$ are motivated by both practical and modeling requirements. 
Numerous applications in robust optimization, including  
inventory planning and project crashing problems, involve 
binary uncertain parameters; see~\cite{Gokalp.Mittal.Hanasusanto.2018}. In this case, 
we can incorporate binary variables in $\U$ via  quadratic 
constraints of the form in~(\ref{equ:uncertainty}). Specifically, we have that
$u_k \in \{0,1\}$ is equivalent to $u_k^2 = u_k$. If the relation
 $0 \leq u_k \leq 1$ is implied by  $\Uzero$  (note that we can explicitly introduce these constraints 
into $\Uzero $ if necessary), then we have $0 = \min_{\uu \in \Uzero } 
\left\{ - u_k^2 + u_k \right\}$, which shows that the  quadratic constraint $-u_k^2 + u_k = 0$ satisfies the condition in 
 Assumption~\ref{ass:quadcons}. Furthermore, these constraints will be crucial for deriving 
our improved decision rules as they enable us to model complementary constraints, e.g., 
$u_{k}u_{k'}  = 0$;  see Section \ref{sec:enhanced} for detail. 
If~$\Uzero$ implies that both $u_{k}$ and $u_{k'}$ are nonnegative 
and bounded, then we have $0 = \min_{\uu \in \Uzero } \left\{ u_{k} u_{k'} \right\}$.
Thus, the quadratic constraint $u_{k} u_{k'}= 0$ satisfies the condition 
in Assumption \ref{ass:quadcons}.

Two-stage robust optimization problems of the form (\ref{equ:tsro}) are generically NP-hard 
\cite{Ben-Tal.Goryashko.Guslitzer.Nemirovski.2004}. A popular conservative approximation 
scheme is obtained in \emph{linear decision rules}, where we restrict the 
recourse action $\bm y(\cdot)$ to be a linear function of $\bm u$. If the problem has 
fixed \emph{recourse} (i.e., $\BB(\uu)$ and  $\bm d(\uu)$ are constant), 
then the linear decision rule approximation leads to 
tractable linear programs. On the other hand, if the problem has \emph{non-fixed recourse} (i.e., $\BB(\uu)$ or $\bm d(\uu)$ depends linearly 
in $\bm u$), then the approximation itself is intractable. In the following, we show that the linear decision rule problems are amenable
to exact copositive programming reformulations. Furthermore, in the specific case where the
problem has fixed recourse, we develop an improved approximation in \emph{quadratic 
decision rules}, and show that the resulting optimization problems can also be reformulated as equivalent copositive programs.

\subsection{Linear decision rule  for problems with non-fixed recourse} \label{sec:ldr}

In this section, we derive an exact copositive program by applying linear decision rules to problem~\eqref{equ:tsro}. 
Instead of considering all possible choices of functions $\yy:\mathcal U\rightarrow \RR^N$ from $\FF_{K+1, \, N}$, we
restrict ourselves to linear functions of the form
\[
\bm y(\uu)=\bm Y\uu,
\]
for some coefficient matrix $\bm Y\in \RR^{N\times (K+1)}$. 
This setting gives rise to the following conservative 
approximation of problem \eqref{equ:tsro}:
\begin{equation} \label{equ:ldr}
\tag{$\LDR$}
\begin{array}{rcll} 
Z^{\LDR}=&\inf & \cc^{\top} \x + \sup \limits_{\uu \in \U} \dd(\uu)^{\top} \left( \bm Y \uu \right)  \\
&\st  &  \A(\uu)\x +  \BB(\uu) \bm Y \uu  \geq \hh(\uu) \ \ \ \forall \, \uu \in \U  \\
 &    & \x \in \X, \, \bm Y\in \RR^{N\times (K+1)}.
\end{array}
\end{equation}
Problem \eqref{equ:ldr} is finite-dimensional but remains difficult to solve as there 
are infinitely many constraints parametrized by $\uu \in \U$. In particular, it is shown 
in \cite{Ben-Tal.Goryashko.Guslitzer.Nemirovski.2004} that the problem is  NP-hard 
via a reduction from the problem of checking matrix copositivity. 

We now show that an equivalent copositive programming reformulation can 
principally be derived for problem~\eqref{equ:ldr}. We first introduce the 
following technical lemmas, which are fundamental for our derivations. 
The first technical lemma establishes the 
equivalence between a nonconvex quadratic program
\begin{equation} \label{equ:qp}
\begin{array}{ll}
\sup &  \uu^\top \bm{\widehat C}_0 \uu \\
\st  & \mathbf{e}_{K+1}^\top \uu = 1 \\
      & \uu^\top \bm{\widehat C}_i \uu = 0 \ \ \ \forall \, i \in [I] \\
      & \uu \in \K 
\end{array}
\end{equation}
and its copositive relaxation
\begin{equation} \label{equ:qp-cp}
\begin{array}{ll}
\sup &   \bm{\widehat C}_0  \bullet \bm{U} \\
\st  & \mathbf{e}_{K+1}\mathbf{e}_{K+1}^\top \bullet \bm U = 1 \\
      & \bm{\widehat C}_i \bullet \bm U = 0 \ \ \ \forall \, i \in [I]  \\
      & \bm U  \in \CP(\K),
\end{array}
\end{equation}
where $\bm{\widehat C}_0  \in \SYM^{K+1}, \K \subseteq \RR^{K+1}$ 
is a closed and convex cone, and $\CP(\K)$ is the cone of completely positive 
matrices with respect to $\K$.

\begin{lemma} [\cite{Burer.2012}, Corollary 8.4, Theorem 8.3] \label{lem:burer}
 Suppose that Assumptions \ref{ass:bounded} and \ref{ass:quadcons} hold.
Then, problem (\ref{equ:qp-cp}) is equivalent to
(\ref{equ:qp}), i.e., i) the optimal value of (\ref{equ:qp-cp}) is equal to that of (\ref{equ:qp});
ii) if $\bm{U}^\star$ is an optimal solution for (\ref{equ:qp-cp}), then $\bm{U}^\star\mathbf{e}_1$
is in the convex hull of optimal solutions for (\ref{equ:qp}).
\end{lemma}

\begin{lemma} \label{lem:zero}
Suppose  Assumption \ref{ass:bounded} holds. Then,  for any  $(\bm z, \tau) \in \K$, 
we have $\tau=0$ implies $\bm z=\bm 0$. 
\end{lemma}
\begin{proof}
See the Appendix. 
\end{proof}

The dual of problem (\ref{equ:qp-cp}) is given by the following linear program over the cone of 
copositive matrices with respect to $\K$: 
\begin{equation} \label{equ:qp-cop}
\begin{array}{ll}
\inf &   \lambda  \\
\st  & \lambda \, \mathbf{e}_{K+1}\mathbf{e}_{K+1}^\top +  \sum\limits_{i=1}^I \alpha_i \bm{\widehat C}_i - \bm{\widehat C}_0 \in \COP(\K)  \\
      & \lambda \in \RR, \ \bm \alpha \in \RR^I.  
\end{array}
\end{equation}
Our next technical lemma establishes strong duality for the primal and dual pair. 
\begin{lemma} \label{lem:interior}
Suppose Assumption~\ref{ass:bounded} holds. Then,
strong duality holds between problems (\ref{equ:qp-cp}) and~(\ref{equ:qp-cop}). 
\end{lemma} 
\begin{proof}
See the Appendix. 
\end{proof}

In the following, we define the auxiliary matrices:
\begin{equation} \label{equ:ABmat}
\bm{\widehat \Theta}_j := 
\begin{pmatrix} \mathbf{e}_j^\top\bm{\widehat A}_1 \\ \vdots \\ \mathbf{e}_j^\top\bm{\widehat A}_{K+1} \end{pmatrix}  \in \RR^{(K+1) \times M}, \ \ \ 
\bm{\widehat \Lambda}_j := \begin{pmatrix} \mathbf{e}_j^\top\bm{\widehat B}_1 \\ \vdots \\ \mathbf{e}_j^\top\bm{\widehat B}_{K+1} \end{pmatrix}  \in \RR^{(K+1) \times N}, \;\;\textup{ and }
\end{equation}
\begin{equation} \label{equ:OmegaxY}
\bm{\Omega}_j\left( \bm x, \,  \bm Y\right) :=  \frac12 \left( \bm{\widehat \Theta}_j \x \mathbf{e}_{K+1}^\top + \mathbf{e}_{K+1}\x^\top \bm{\widehat \Theta}_j^\top +
\bm{\widehat \Lambda}_j \bm{Y} + \bm{Y}^\top\bm{\widehat \Lambda}_j^\top - \bm{\widehat h}_j\mathbf{e}_{K+1}^\top - \mathbf{e}_{K+1}\bm{\widehat h}_j^\top \right) \ \ \forall \, j \in [J],
\end{equation}
where $\mathbf{e}_j$ represents the $j$th standard basis vector in $\RR^J$. 
We are now ready to state our main result. 
\begin{theorem} \label{the:ldr-ref}
Problem (\ref{equ:ldr}) is equivalent to the copositive program
\begin{equation} \label{equ:cop_ldr}
\begin{array}{rcll} 
Z^{\LDR}=&\inf  & \cc^{\top} \x + \lambda \\
&\st  & \lambda \, \mathbf{e}_{K+1}\mathbf{e}_{K+1}^\top -  \dfrac12 \left( \mathbf{\widehat D}^\top\bm{Y} + \bm{Y}^\top \mathbf{\widehat D} \right)  + \sum\limits_{i=1}^I \alpha_i \bm{\widehat C}_i \in \COP(\K) \\
&     &  \bm{\Omega}_j\left(\bm x, \, \bm Y\right)   - \pi_j \, \mathbf{e}_{K+1}\mathbf{e}_{K+1}^\top - \sum\limits_{i=1}^I  [\bm{\beta}_j]_i \, \bm{\widehat C}_i  \in \COP(\K) \ \ \ \forall \, j \in [J] \\     
&     &\x \in \mathcal X,\, \  \lambda \in \RR,\,  \ \bm{Y} \in \RR^{N \times(K+1)}, \ \bm \pi \in \RR_+^J, \ \bm \alpha \in \RR^I, \  \bm \beta_j \in \RR^I \quad \forall \, j \in [J], 
\end{array}
\end{equation}
where the affine functions $\bm{\Omega}_j(\bm x, \bm Y)$, $ j \in [J]$, are defined as in (\ref{equ:OmegaxY}).  
\end{theorem}
\begin{proof}
See the Appendix. 
\end{proof}


\subsection{Quadratic decision rules for problems with fixed recourse} \label{sec:qdr}

We now study two-stage robust optimization problems with fixed recourse. In this simpler setting, the second-stage cost coefficients and the recourse matrix are deterministic, i.e., 
\[
\begin{array}{l}
\dd(\uu)=\bm{\widehat d}\in \RR^N\quad \textup{ and } \quad \BB(\uu)=\bm{\widehat B} \in \RR^{J\times N}\qquad \forall \, \uu \in \RR^{K+1}.
\end{array}
\]
Using techniques developed in the previous section, we will derive a copositive 
programming reformulation by applying decision rules to the recourse action 
$\bm y:\mathcal U\rightarrow\RR^N$. Since $\dd(\uu)$ and $\BB(\uu)$ are 
constant, we may  utilize the more powerful quadratic decision rules defined as
\[
[\bm y(\uu)]_n = \uu^\top\bm Q_n\uu \qquad \forall \, n \in [N],
\]
for some coefficient matrices $\bm Q_n\in \SYM^{K+1}$, $n \in [N]$. 
This yields the following conservative approximation of problem \eqref{equ:tsro}:
\begin{equation} \label{equ:qdr-ref}
\tag{$\QDR$}
\begin{array}{rcll} 
Z^{\QDR}=&\inf  & \cc^{\top}\x +\sup\limits_{\uu \in \U} \sum\limits_{n=1}^N  \widehat d_n \uu^\top \bm{Q}_n \uu  \\
&\st  &  \uu^\top\bm{\widehat \Theta}_j \x +  \sum\limits_{n=1}^N  \widehat b_{j n} \uu^\top\bm{Q}_n\uu \geq \bm{\widehat h}_j^\top \uu  \ \ \ \forall \, \uu \in \U \ \forall \, j \in [J] \\
&& \x \in \X,\;\bm Q_n \in \SYM^{K+1} \ \ \ \forall \, n \in [N]. 
\end{array}
\end{equation}
In view of the restriction $u_{K+1} = 1$ in the description of $\U$, 
the decision rule $[\yy(\uu)]_n = \uu^\top\bm{Q}_n\uu$ constitutes a 
homogenized version of a non-homogenized quadratic function in the 
primitive vector $(u_1, \ldots, u_K)^\top$. We remark that optimizing for the 
best quadratic decision rule is generically NP-hard \cite[Section 14.3.2]{Ben-Tal.Ghaoui.Nemirovski.2009}. 
This strongly justifies our proposed copositive programming reformulation, which we derive in the following theorem.   To that end, we define the affine functions
\begin{equation} \label{equ:GxQ}
\bm{\Gamma}_j\left( \bm x, \bm Q_1, \ldots, \bm Q_N \right) :=  \dfrac12 \left( \bm{\widehat \Theta}_j \x \mathbf{e}_{K+1}^\top + \mathbf{e}_{K+1}\x^\top \bm{\widehat \Theta}_j^\top - \mathbf{e}_{K+1}\bm{\widehat h}_j^\top - \bm{\widehat h}_j \mathbf{e}_{K+1}^\top \right) + \sum \limits_{n=1}^N \,  \widehat b_{jn} \bm{Q}_n\qquad\forall j\in[J].
\end{equation}
\begin{theorem} \label{thm:qp} 
Problem (\ref{equ:qdr-ref}) is equivalent to the copositive program
\begin{equation} \label{equ:cop_qdr}
\begin{array}{rcl}
Z^{\QDR}&=\min  & \cc^\top \x + \lambda \\
&\st & \lambda \, \mathbf{e}_{K+1}\mathbf{e}_{K+1}^\top - \sum \limits_{n=1}^N \, \widehat d_n \bm{Q}_n + \sum\limits_{i=1}^I  \alpha_i \bm{\widehat C}_i \in \COP(\K) \\
&    & \bm{\Gamma}_j\left( \bm x, \bm Q_1, \ldots, \bm Q_N\right) - \pi_j \, \mathbf{e}_{K+1}\mathbf{e}_{K+1}^\top - \sum\limits_{i=1}^I  [\bm \beta_j]_i \,  \bm{\widehat C}_i  \in \COP(\K)  \ \ \ \forall \, j \in [J] \\
&    & \x \in \X,\ \lambda \in \RR, \  \bm \alpha \in \RR^I, \ \bm{\pi} \in\RR_+^J, \  \bm Q_n\in\SYM^{K+1} \ \ \forall \, n \in [N], \ \bm \beta_j \in \RR^I \ \ \forall \, j \in [J],
\end{array}
\end{equation}
where the affine functions $ \bm{\Gamma}_j\left( \bm x, \bm Q_1, \ldots, \bm Q_N\right)$,  $j \in [J]$, are defined as in (\ref{equ:GxQ}).
\end{theorem}
\begin{proof}
See the Appendix.
\end{proof}

\subsection{Enhanced decision rules} \label{sec:enhanced}

In this section, we tighten the basic decision rule approximations by employing piecewise linear and piecewise quadratic decision rules. 
While piecewise quadratic decision rules are new concept, piecewise linear decision rules have been studied extensively in the 
literature~\cite{Chen.Zhang.2009,Georghiou.Wiesemann.Kuhn.2015}. Their utilization is supported by a strong theoretical justification: 
For problems with fixed recourse, the optimal recourse action $\bm y(\cdot)$ can be described by a piecewise linear continuous 
function~\cite{Bemporad.Borrelli.Morari.2003}. However, optimizing for the best piecewise linear decision rule is NP-hard even if the 
folding directions and their respective breakpoints are prescribed a priori~\cite[Theorem 4.2]{Georghiou.Wiesemann.Kuhn.2015}. 
As such, one has to rely on another layer of tractable conservative approximation. Unfortunately, the state-of-the-art  approaches are 
futile even in the simplest robust optimization settings (see Example \ref{exa:partition} below). Here, we endeavor to derive tighter 
approximations using copositive programming. 

To this end, for a prescribed number of pieces~$L$, we define the  mappings
\begin{equation}
\label{eq:pw_mappings}
F_\ell(\bm u)=\max\{0,\bm f_\ell^\top \uu  \} \qquad \forall \, \bm u \in \RR^{K+1} \ \forall \, \ell \in [L].
\end{equation} 
Here, $\bm f_\ell:=(\bm g_\ell,-h_\ell)\in\RR^{K+1}$, where  $\bm g_\ell\in\RR^{K}$ denotes the folding direction of the $\ell$-th 
mapping, while $h_\ell$ defines its breakpoint. These mappings constitute the building blocks of our improved decision rules. Specifically, 
by applying the basic linear and 
quadratic decision rules on the lifted uncertain parameter vector 
$\bm v:=(F_1(\bm u), \ldots, F_L(\bm u), \bm u) \in \RR^{L+K+1}$, we arrive at the desired piecewise linear and piecewise quadratic decision rules, respectively. 

\begin{example}[Integer Programming Feasibility Problem]
	\label{exa:IP_Feas}
	Consider a norm maximization problem given by $\max_{\bm u\in\U} \|\bm u\|_1$, 
	where $\U = \{ \bm u \in \RR^K: \bm P \bm u \leq \bm q\} \subseteq[-1,1]^K$ is a 
	prescribed polytope.  
	An elementary analysis shows that the optimal 
	value of this problem is equal to $K$ if and only if there exists a binary vector $\bm u \in\{-1,1\}^{K}$ 
	within the polytope $\U$. Thus, it solves the NP-hard Integer Programming (IP) feasibility 
	problem~\cite{Garey.Johnson.1979}.  We can reformulate the norm maximization problem 
	as a two-stage robust optimization problem, without a first-stage decision $\bm x$, given by
	\begin{equation*} 
	\begin{array}{cl} 
	\inf &  \sup \limits_{\uu \in \U} \displaystyle\mathbf e^\top\yy(\uu) \\
	\st  & \yy(\uu)\geq \bm  u, \; \yy(\uu) \geq  -\bm u  \  \  \  \forall \, \uu \in \U  \\
	&\yy \in \FF_{K, \, K}.
	\end{array}
	\end{equation*}
	Indeed, at optimality we have $[\bm y(\bm u)]_k = |u_k|$, which implies that $\mathbf e^\top\yy(\uu)=\|\bm u\|_1$.  
	Consider now the mappings 
	\begin{equation*}
	F_\ell(\bm u) = \max\{ 0, \, u_\ell \} \qquad \forall \, \bm u \in \RR^{K} \ \forall \, \ell \in [K].
	\end{equation*} 
	Our previous argument shows that the piecewise linear decision rule given by
	\begin{equation*}
	[\bm y(\bm u)]_\ell= -u_\ell + 2F_\ell(\bm u) = -u_\ell + \max\{ 0, \, 2u_\ell \} = |u_\ell|\qquad\forall \ell\in[K]
	\end{equation*}
	is optimal. This decision rule is linear in the lifted parameter 
	vector $(F_1(\bm u),\dots,F_K(\bm u),\bm u)$.  
\end{example}

To formalize the idea into our setting, we define the lifted set 
\begin{equation}
\label{eq:uncertainty_expand}
\Ubar := \left\{
\bm v := (\bm w, \uu) \in \RR^{L}\times \U :~
w_\ell = F_\ell(\uu) \quad \forall \,  \ell \in [L]
\right\},
\end{equation}
and  
the lifted parameters
\[
\begin{array}{c}
\A'\left(\bm v\right) = \A\left(\bm u\right), \ \ \  \BB'\left(\bm v\right)= \BB\left(\bm u\right), \ \ \ 
\dd'\left(\bm v\right)=  \dd\left(\bm u\right), \ \ \  
\hh'\left(\bm v\right) =   \hh\left(\bm u\right),\\[4mm]
\bm{\widehat \Theta}'_j=\begin{pmatrix}\bm 0^\top, \bm{\widehat \Theta}_j^\top \end{pmatrix}^\top \in\RR^{(L+K+1)\times M}\qquad \forall j\in[J]. 
\end{array}
\]
Then, by replacing the set $\U$ with $\Ubar$ and employing the above lifted parameters 
in~\eqref{equ:ldr} and \eqref{equ:qdr-ref}, we obtain  the corresponding piecewise decision rule problems. These are   given by  

\begin{equation} \label{equ:pldr}
\tag{$\PLDR$}
\begin{array}{rcll} 
Z^{\PLDR}=&\inf & \cc^{\top} \x + \sup \limits_{\bm v\in \Ubar} \dd'({\vv})^{\top}  \bm Y \vv   \\
&\st  &  \A'\left(\vv\right)\x +  \BB'\left(\vv\right) \bm Y \vv  \geq \hh'\left(\vv\right) \ \ \ \forall \, \bm v\in \Ubar  \\
&    & \x \in \X, \, \bm Y\in \RR^{N\times (L+K+1)}
\end{array}
\end{equation}
and
\begin{equation} \label{equ:pqdr}
\tag{$\PQDR$}
\begin{array}{rcll} 
Z^{\PQDR}=&\inf  & \cc^{\top}\x +\sup\limits_{\uu'\in \Ubar} \sum\limits_{n=1}^N  \widehat d_n {\vv}^\top \bm{Q}_n \vv  \\
&\st  &  \vv^\top\bm{\widehat \Theta}'_j \x +  \sum\limits_{n=1}^N  \widehat b_{j n} {\vv}^\top\bm{Q}_n\vv \geq\left[\hh'\left(\vv\right) \right]_n  \ \ \ \forall \, \vv \in \Ubar \ \forall \, j \in [J] \\
&& \x \in \X,\;\bm Q_n \in \SYM^{L+K+1} \ \ \ \forall \, n \in [N],
\end{array}
\end{equation}
respectively.

We now establish that the piecewise decision rule problems can be equivalently reformulated as polynomial size copositive programs. 
The   reformulations leverage  our capability to incorporate complementary constraints in the uncertainty set $\U$.  
 We remark that the problems~\eqref{equ:pldr} and \eqref{equ:pqdr} share the same structure as their plain vanilla counterparts~\eqref{equ:ldr} and \eqref{equ:qdr-ref}. 
To establish that equivalent copositive programs can also be derived for these problems, we need to show  that the set $\Ubar$ can be brought into the standard form~\eqref{equ:uncertainty}.
First, we prove that the non-convex set $\Ubar$ is equivalent to a concise set involving~$\mathcal O(L)$ linear and complementary constraints. 
\begin{theorem}
	\label{prop:EDR} 
	The lifted uncertainty set in \eqref{eq:uncertainty_expand} can be represented as 
	the set
	\begin{equation}
	\label{eq:uncertainty_expand_}
	\Ubar = \left\{ (\bm w, \uu) \in \RR^{L}\times\U : 
	\begin{array}{l} 
	 \bm 0\leq  \bm w \leq \overline{\bm w} \\
	w_\ell \geq \bm f_\ell^\top \bm u 
	\quad \forall \,  \ell \in[L] \\
	w_\ell(w_\ell - \bm f_\ell^\top \bm u)= 0 \quad \forall \, \ell \in[L]
	\end{array}
	\right\},
	\end{equation}
	where $\overline{\bm w} \in \RR^L$  is a vector whose components are upper bounds on the auxiliary parameters $w_1,\ldots,w_L$. 
	These upper bounds can be computed by solving $L$ convex conic optimization problems given by 
	\begin{equation*} \label{equ:upperw}
	\overline{w}_\ell := \max\limits_{ \uu \in \Uzero}  \;\;\bm f_\ell^\top \uu \qquad\forall \ell\in[L], 
	\end{equation*}
	where $\Uzero$ is defined as in (\ref{equ:Lset}).
\end{theorem}
\begin{proof}
	For any fixed $\bm u \in \mathcal U$ and $\ell \in [L]$, the complementary constraint 
	$w_\ell(w_\ell-\bm f_\ell^\top \bm u) = 0$ implies that either $w_\ell=0$ or 
	$w_\ell=\bm f_\ell^\top\bm u $. Thus, the constraints $w_\ell\geq 0$ and 
	$w_\ell \geq \bm f_\ell^\top\bm u$ yield $w_\ell = \max\{ 0, \bm f_\ell^\top \bm u\}$. 
	This completes the proof. 
\end{proof}

\begin{remark}
The inclusion of the upper bound $\overline{\bm w}$ on the lifted parameters $\bm w$ ensures
the boundedness of  the uncertainty set $\U'$, which is required by Assumption \ref{ass:bounded}. 
We also note that $\U'$ is closed since $w_\ell(w_\ell - \bm f_\ell^\top \bm u)= 0  \Leftrightarrow w_\ell =0$ or $w_\ell - \bm f_\ell^\top \bm u = 0$ for all $\ell \in [L]$. 
Therefore, $\U'$ is a compact set. 
\end{remark}

\noindent Next, in view of the equivalent  set in \eqref{eq:uncertainty_expand_}, we define the lifted cone
\[
\K' := \left\{ (\bm w, \uu) \in \RR^{L}  \times \U :
\begin{array}{l} 
\bm 0\leq  \bm w \leq \overline{\bm w}u_{K+1} \\
w_\ell \geq \bm f_\ell^\top \bm u 
\quad \forall \,  \ell \in[L]
\end{array}
\right\}.
\]
Letting the matrices $\bm{\widehat C}_\ell$, $\ell\in[L]$, be defined as 
\begin{equation*}
\bm{\widehat C}_\ell=(\mathbf e_\ell^\top,\bm 0^\top)^\top(\mathbf e_\ell^\top,\bm 0^\top) - \frac{1}{2}(\mathbf e_\ell^\top,\bm 0^\top)^\top(\bm 0^\top,\bm f_\ell^\top)-\frac{1}{2}(\bm 0^\top,\bm f_\ell^\top)^\top(\mathbf e_\ell^\top,\bm 0^\top) \qquad \forall\ell\in[L], 
\end{equation*}
we can  capture the  complementarity constraints in $\Ubar$ via the quadratic equalities $\vv^\top\bm{\widehat C}_\ell\vv = 0$, $\ell \in [L]$. Thus, the lifted set coincides with the set
%
\begin{equation*}
\label{eq:uncertainty_expand_homo}
\U' := \left\{
\vv := (\bm w, \uu) \in \K' :
 \begin{array}{l}
 u_{K+1} = 1, \ \vv^\top\bm{\widehat C}_\ell\vv = 0 \ \ \forall \, \ell \in [L]
\end{array}
\right\},
\end{equation*}
which indeed assumes the standard form in \eqref{equ:uncertainty}. In summary, we have 
established that equivalent copositive programs can  be derived for the proposed piecewise 
linear and piecewise quadratic decision rule problems. As described in Section \ref{sec:SDP_approx}, 
tractable semidefinite programming approximations can then be obtained by replacing the cone $\COP(\K')$ 
in the respective copositive programs with the  inner approximation $\IA(\K')$.

\section{ Semidefinite programming solution schemes}
\label{sec:SDP_approx}
Our equivalence results indicate that the decision rule problems are amenable to semidefinite 
programming solution schemes. Specifically, there exists a hierarchy of 
increasingly tight semidefinite-representable inner approximations that converge to 
$\COP(\K)$~\cite{Bomze.Klerk.2002,DeKlerk.Pasechnik.2002,Lasserre.2009,Parrilo.2000}.
Replacing the cone $\COP(\K)$  with these inner approximations gives rise to  conservative 
semidefinite programs that can be solved using standard off-the-shelf solvers. 
In this section, we develop new tractable  approximations and exact semidefinite reformulations for the 
 copositive programs derived  in Section \ref{sec:two-stage}. To this end, we  
primarily consider polyhedral- and second-order cone-representable uncertainty sets defined 
via closed and convex cones of the following generic form:
\begin{equation} \label{equ:spec_cone}
\K :=\left \{ 
\uu \in \RR^{K} \times \RR_+ : 
\bm{\widehat P} \uu \geq 0, \   \bm{\widehat R} \uu \in \SOC(K_r) 
\right \},
\end{equation}
with $\bm{\widehat P} \in \RR^{K_p \times (K+1)}$ and $\bm{\widehat R} \in \RR^{K_r \times (K+1)}$. 
As illustrated in the examples of Section \ref{sec:two-stage}, the above generic structure for the cone 
$\K$ can encompass many commonly used uncertainty sets in practice.

\subsection{Conservative approximations} 

We consider a  semidefinite-representable approximation  to the  cone $\COP(\K)$ given by
\begin{equation} \label{equ:innercone}
\IA(\K) := \left\{ \bm{V} \in \SYM^{K+1} : 
\begin{array}{l} 
\bm{W} \in \SYM^{K+1}, \ \bm{W} \succeq \bm 0, \ \bm{\Sigma} \in \SYM^{K_p} \\ 
\bm{\Psi} \in \SYM^{K+1}, \ \bm{\Phi} \in \RR^{K_p \times K_r}, \  \tau \in \RR  \\
\bm{V} = \bm{W} + \tau \bm{\widehat S} + \bm{\widehat P}^\top\bm{\Sigma}\bm{\widehat P} + \bm{\Psi}, \  \bm{\Sigma} \geq \bm 0, \ \tau \geq 0 \\  
\bm{\Psi} = \frac12(\bm{\widehat P}^\top \bm{\Phi} \bm{\widehat R} +  \bm{\widehat R}^\top \bm{\Phi}^\top\bm{\widehat P}), \ \text{Rows}(\bm{\Phi}) \in \SOC(K_r) 
\end{array} \right\},
\end{equation}
where the matrix $\bm{\widehat S}$ is defined as 
\begin{equation} \label{equ:Shat}
\bm{\widehat S} :=  \bm{\widehat R}^\top \mathbf{e}_{K_r} \mathbf{e}_{K_r}^\top\bm{\widehat R}  - \sum\limits_{\ell = 1}^{K_r-1}\bm{\widehat R}^\top \mathbf{e}_\ell \mathbf{e}_\ell^\top\bm{\widehat R}. 
\end{equation}
We  now  establish that $\IA(\K)$ is a subset of $\COP(\K)$.\footnote{Hence, we use the abbreviation ``$\IA$"  which stands for~\emph{``Inner Approximation."}} 
To this end, we make the following observation. 
\begin{lemma} \label{lem:linear_soc}
We have $\uu^\top\bm{\widehat S}\uu \geq 0$ for all $\uu \in \K$.
\end{lemma}
\begin{proof}
See the Appendix. 
\end{proof}
\noindent Using Lemma \ref{lem:linear_soc}, we are now ready to prove the containment result. 
\begin{proposition} \label{prop:ia-cop}
We have 
$\IA(\widehat \U) \subseteq \COP(\widehat \U)$.
\end{proposition}
\begin{proof}
See the Appendix. 
\end{proof}

\noindent Replacing the cone $\COP(\K)$ in (\ref{equ:cop_ldr}) and (\ref{equ:cop_qdr}) with the inner approximation $\IA(\K)$ gives rise to
 conservative semidefinite programs. We denote their optimal values as $Z^{\LDR}_{\textup{IA}}$ and $Z^{\QDR}_{\textup{IA}}$, respectively. 
The following proposition summarizes our current findings. 
\begin{proposition}
We have $Z^{\LDR}\leq Z^{\LDR}_{\textup{IA}}$ and $Z^{\QDR}\leq Z^{\QDR}_{\textup{IA}}$. 
\end{proposition} 

An alternative conservative approximation  scheme is proposed by Ben-Tal et al.~
in view of the~\emph{approximate $S$-lemma} \cite[Theorem B.3.1]{Ben-Tal.Ghaoui.Nemirovski.2009}. 
In this case, the corresponding inner approximation   for the cone $\COP(\K)$ is given by 
\begin{equation} \label{equ:ascone}
{\AS}(\K) := \left\{  \bm{V} \in \SYM^{K+1} : 
\begin{array}{l}
\tau \geq 0, \ \bm \theta \in \RR_+^{K_p}, \ \bm W \in \SYM^{K+1}, \ \bm W \succeq \bm  0  \\ 
\bm{V} =  \bm W + \tau \bm{\widehat S}  +  \dfrac12\left( \bm{\widehat P}^\top \bm \theta \mathbf{e}_{K+1}^\top + \mathbf{e}_{K+1}\bm \theta^\top \bm{\widehat P} \right)
\end{array}
\right\},
\end{equation}
where $\bm{\widehat S}$ is defined as in (\ref{equ:Shat}).
Replacing the cone $\COP(\K)$ in (\ref{equ:cop_ldr}) and (\ref{equ:cop_qdr}) with $\AS(\K)$ yields 
conservative semidefinite programs whose optimal values are denoted as $Z^{\LDR}_{\textup{AS}}$ and $Z^{\QDR}_{\textup{AS}}$, respectively. 
We now show that $\AS(\K)$ is inferior to $\IA(\K)$ for approximating $\COP(\K)$. 

\begin{proposition} \label{lem:ia-as}
We have  ${\AS}(\K) \subseteq \IA(\K)$.
\end{proposition}
\begin{proof}
The inclusion follows by simply  
setting $\bm{\Sigma} =  \frac12 (\bm \theta \mathbf{e}_{K+1}^\top + \mathbf{e}_{K+1}\bm \theta^\top)$
and $\bm\Psi = \bm 0$ in $\IA(\K)$. 
\end{proof}

Lastly, another conservative approximation scheme naturally arises in polynomial decision 
rules~\cite{Bertsimas.Iancu.Parrilo.2010}. Here, one first imposes the restriction that the recourse 
function $\bm y(\cdot)$ in \eqref{equ:tsro} is a polynomial  of fixed degree $d$.  Since optimizing 
for the best  polynomial decision rule is generically NP-hard, one resorts to another layer of approximation in  
semidefinite programming. To this end, consider a degree~$d$ polynomial decision rule. For  problems with non-fixed recourse we find that each 
semi-infinite constraint in \eqref{equ:tsro} reduces to the problem of checking  the non-negativity of a polynomial 
of degree $\hat d=d+1$ over  the set $\U$, while for problems with fixed recourse it reduces to the problem of checking the non-negativity a 
polynomial of degree $\hat d=d$ over the set~$\U$. A sufficient condition  
would be if the polynomial admits a sum-of-squares (SOS) decomposition relative to $\U$, which is equivalent 
to checking the feasibility of a semidefinite-representable constraint system whose size grows exponentially 
in $d$. We refer the reader to \cite{Bertsimas.Iancu.Parrilo.2010} for a more detailed discussion about the 
SOS decomposition and its parametrization. When the corresponding polynomial in the semi-infinite constraint is of 
degree $\hat d=2$, then one can show the resulting constraint system coincides with that from the 
approximate S-lemma. To this end, let $Z^{\mathcal P_d}_{\textup{SOS}}$ be the optimal value of the 
approximation when polynomial decision rules of degree $d$ are employed. Then, we have $Z^{\mathcal P_1}_{\textup{SOS}}=Z^{\LDR}_{\textup{AS}}$ and $Z^{\mathcal P_2}_{\textup{SOS}}=Z^{\QDR}_{\textup{AS}}$.  
Increasing the degree of the polynomial decision rules helps improve approximation quality at the expense of significant 
computational burden and numerical instability, even if we merely raise the degree by $1$ (that is, when we employ 
quadratic decision rules for problems with non-fixed recourse or  cubic decision rules for problems with fixed recourse). 

The findings of this section culminate in the following theorem. 
\begin{theorem}
The following chains of inequalities hold:
\begin{equation*}
Z^{\LDR} \leq Z^{\LDR}_{\textup{IA}} \leq Z^{\LDR}_{\textup{AS}}=Z^{\mathcal P_1}_{\textup{SOS}}\quad\textup{ and }\quad Z^{\QDR}\leq  Z^{\QDR}_{\textup{IA}}\leq Z^{\QDR}_{\textup{AS}}=Z^{\mathcal P_2}_{\textup{SOS}}. 
\end{equation*}
\end{theorem}

\subsection{Exact reformulations} 

We identify two cases where the semidefinite-based approximations are equivalent to the respective copositive programs.  
Firstly, in view 
the \emph{exact} $S$-lemma, one can show that the inner approximation  $\IA(\K)$ coincides with $\COP(\K)$ whenever the cone 
$\K$ in (\ref{equ:spec_cone}) is described  by only a second-order cone 
constraint $\bm{\widehat R}\uu \in \SOC(K_p)$. 
\begin{proposition}[$S$-Lemma] \label{lem:soc}
If $\K = \{ \uu \in \RR^{K+1} : \bm{\widehat R}\uu \in \SOC(K_p) \}$ then
\[
\COP(\K) = \IA(\K) = \AS(\K):= \left\{ \bm{V} \in \SYM^{K+1} : \bm{V} \succeq  \tau \bm{\widehat S} ,\ \tau \geq 0  \right\},
\]
where $\bm{\widehat S} \in \SYM^{K+1}$ is defined as in (\ref{equ:Shat}). 
\end{proposition}

Another exactness result arises when linear constraints  are present in $\K$ and they satisfy the following condition:
\begin{assumption} \label{ass:caps}
If $\uu \in \RR^{K+1}$ satisfies $\bm{\widehat R}\uu \in \SOC(K_p)$ and 
$\bm{\widehat p}_\ell^\top \uu = 0$ for some $\ell \in [K_p]$, then $\bm u\in\K$. %
\end{assumption}
\noindent The condition stipulates that the cone  $\{\bm u\in\RR^{K+1}:\bm{\widehat R}\uu \in \SOC(K_p)\}$ 
must not contain points in the hyperplane $\bm{\widehat p}_\ell^\top \uu = 0$ that do not not belong to $\K$. 
Applying the restriction  $u_{K+1}=1$, we find that the implied uncertainty set for the primitive vector 
$(u_1,\ldots,u_K)^\top$ is given by an intersection of a ball and a polytope whose facets  do not intersect within the ball.  

\begin{example}
Consider the set 
\[
\U : = \left\{  \uu \in \RR^2 \times \{1\} : u_1^2 + u_2^2 \leq 1, \ u_1 \geq -\frac12, \ u_1 \leq \frac12 \right \}.
\]
The two lines $u_1 = -\frac12$ and $u_1 = \frac12$   do not intersect as they are parallel. 
Thus, Assumption \ref{ass:caps} holds for this uncertainty set. 
\end{example}

\noindent We state the second exactness result in the following proposition. 
\begin{proposition} [Theorem 5 in \cite{Burer.2015}] \label{lem:burer-yang}
If Assumption~\ref{ass:caps} holds then $\COP(\K) = \IA(\K)$. 
\end{proposition}
\noindent We remark that this positive result holds only for the proposed inner approximation $\IA(\K)$, 
and not for the cone $\AS(\mathcal K)$ which is obtained from applying the approximate S-lemma. 
Thus, in general we may still  have $ \AS(\K)\subseteq \COP(\K)$.  

We conclude the section with the following theorem regarding the exactness of the semidefinite programs. 
\begin{theorem}
If the cone $\K$ is given by $\{ \uu \in \RR^{K+1} : \bm{\widehat R}\uu \in \SOC(K_p) \}$ 
or if it satisfies Assumption~\ref{ass:caps} then  $Z^{\LDR}_{\textup{IA}}  = Z^{\LDR}$ and $Z^{\QDR}_{\textup{IA}}  =Z^{\QDR}$. 
\end{theorem}

\subsection{Approximation quality of the enhanced decision rules} 
We now restrict our study to the case of two-stage robust optimization problems with fixed recourse and  
with piecewise linear decision rules. In this setting, linear programming approximations have been 
proposed for the decision rule problems~\cite{Chen.Zhang.2009,Georghiou.Wiesemann.Kuhn.2015}. 
If in addition  the uncertainty set $\mathcal U$ is given by a hyperrectangle and each folding 
direction $\bm g_\ell$ is aligned with a coordinate axis, then these 
linear programs become exact~\cite{Georghiou.Wiesemann.Kuhn.2015}.
Unfortunately, for generic uncertainty sets  the resulting approximation can sometimes be of poor quality. 

\begin{example}[Partition Problem]
	\label{exa:partition}
	Consider  the following instance of IP feasibility problem (Example~\ref{exa:IP_Feas}), which corresponds to the NP-hard partition problem. 
	Given an input vector $\bm c \in \mathbb N^{K}$, the problem asks if one can partition 
	the components of $\bm c$ into two sets so that both sets have an equal sum. We can 
	reduce this problem to the instance of IP Feasibility problem that seeks for a binary 
	vector $\bm u \in\{-1,1\}^{K}$   within the polytope $\mathcal U=\{\bm u \in [-1,1]^K :  
	\bm c^\top  \bm u = 0\}$. If a partition exists then the components of $\bm u$ will denote 
	the indicator function of 
	the two sets. For example, if $\bm c = (1, 2 , 3)^\top$ then the possible solutions are 
	$\bm u=(1, 1, -1)^\top$ or  $\bm u= (-1, -1, 1)^\top$. On the other hand, 
	if $\bm c = (2, 2, 3)^\top$, then no such solution exists and necessarily the optimal 
	value of the corresponding norm maximization problem is strictly less than $K=3$. In particular, one can show that the optimal value is $2.5$, which is attained              by the solution $\bm u = (0.5, 1, 1)^\top$. 
	
	For the input $\bm c = (2, 2, 3)^\top$, the best piecewise linear decision rule approximation 
	in the literature yields a conservative upper 
	bound of $3$, which fails to certify the non-existence of binary solutions. On the other hand, 
	the semidefinite programming approximation of the equivalent copositive program yields a tighter upper bound of $2.54$, and thus provides a 
	correct certificate. As the corresponding two-stage problem has 
	fixed recourse, our scheme 
	allows to utilize quadratic decision rules. In this case, the resulting semidefinite program 
	yields the best optimal value of $2.5$. 
\end{example}
\noindent The above example highlights the surprising fact that, even for seemingly trivial low-dimensional problem instances,  one necessarily has to go through the copositive programming route in order to obtain a satisfactory approximation for the piecewise decision rule problem. 

We now formally establish that the  semidefinite programming approximation  obtained from 
applying  piecewise linear decision rules is never inferior to the state-of-the-art scheme by 
Georghiou et al.~\cite{Georghiou.Wiesemann.Kuhn.2015}. In the following, we briefly discuss their setting and formulate the corresponding  lifted uncertainty set $\Ubar$. For a cleaner exposition, we  primarily consider the setting of piecewise linear decision rules with axial segmentation where each folding direction is aligned with a coordinate axis. We remark that all results extend to the case  with general segmentation, albeit at the expense of more  cumbersome notation (see Section 4.2~of \cite{Georghiou.Wiesemann.Kuhn.2015}).  To this end,  let the interval 
$[\underline u_k,\overline u_k]$ be the marginal support of the $k$-th uncertain parameter.
For each coordinate axis $u_k$, we  generate~$L$ piecewise linear mappings  in view of 
prescribed breakpoints $h_{k,1}=\underline u_k<h_{k,2}<\ldots<h_{k,L}<\overline u_k$, 
as follows: 
	\begin{equation}
	\label{eq:mappings_Georghiou}
	\tilde F_{k,l}(\bm u)=
	\max\{0,u_k-h_{k,\ell}\}-\max\{0,u_k-h_{k,\ell+1}\}\qquad\forall \ell\in[L].
	\end{equation}
	To simplify the notation, we assume that there are exactly $L$ mappings for each coordinate axis.  
	Such a construction gives rise to the lifted uncertainty set 
	\begin{equation}
	\label{eq:expanded_Georghiou}
	\Ubar := \left\{
	(\bm w, \uu) \in \RR^{KL}\times\U :
	w_{k,\ell} = \tilde F_{k,\ell}(\uu) \quad \forall \, k\in[K]\  \ell \in [L]
	\right\}.
	\end{equation}
         Note that each mapping in \eqref{eq:mappings_Georghiou} can  be defined through the 
	difference  $\tilde F_{k,l}(\bm u)=F_{k,l}(\bm u)-F_{k,{l+1}}(\bm u)$, 
	where the functions $F_{k,l}(\bm u)=\max\{0,\bm f_{k,\ell}^\top\bm u\}$, $\ell \in[L]$, assume the standard form described 
	in~\eqref{eq:pw_mappings},
	with $\bm f_{k,\ell}=(\mathbf e_k,-h_{k,\ell})$, $\ell\in[L]$. By our construction of $\U'$, we can further impose that $F_{k,1}(\bm u)=u_k-\underline u_k$ and $F_{k,L+1}(\bm u)=0$. 
	
	Using Theorem \ref{prop:EDR}, the lifted set in \eqref{eq:expanded_Georghiou} can be reformulated as
	\begin{equation} \label{eq:expanded_Georghiou_}
	\Ubar = \left\{
	(\bm w, \uu) \in \RR^{KL} \times  \mathcal U :
	\begin{array}{ll}
	\bm z \in\RR_+^{K(L+1)}  \\  
	w_{k,\ell} = z_{k,\ell}-z_{k,\ell+1}& \forall k\in[K] \  \ell \in [L]\\
	z_{k,1}=u_{k}-\underline u_{k},\;z_{k,L+1}=0 & \forall \, k\in[K]\\
	\ z_{k,\ell} \geq u_k - h_{k,\ell} ,\   \overline{u}_k \geq z_{k,\ell} & \forall \, k\in[K] \  \ell \in [L+1]  \\
	z_{k,\ell}(z_{k,\ell} - u_k+h_{k,\ell})= 0 & \forall \, k \in [K]\  \ell \in [L+1]
	\end{array}
	\right\}.
	\end{equation}
	In view of our discussion in Section \ref{sec:enhanced}, an equivalent copositive
	program can thus be derived for the piecewise linear decision rule problem~\eqref{equ:pldr}. 
	We denote by $Z^{\mathcal{PL}}_{\textup{IA}}$  the optimal value of the corresponding semidefinite 
	programming approximation. Alternatively, in~\cite{Georghiou.Wiesemann.Kuhn.2015}, 
	a tractable outer approximation of $\Ubar$ is derived as follows:
	\begin{equation}
{\U}^{**} = \left\{
(\bm w, \uu) \in \RR^{KL}  \times \U :
\begin{array}{ll}
\displaystyle u_k-\underline u_k = \sum_{\ell\in[L]}w_{k,\ell}& \forall k\in[K] \\
h_{k,2}-\underline u_k \geq w_{k,1} & \forall k\in[K]  \\
(h_{k,\ell+1}-h_{k,\ell})w_{k,\ell-1} \geq(h_{k,\ell}-h_{k,\ell-1})w_{k,\ell} & \forall \, k\in[K]\  \ell \in [L]\setminus \{1\}
\end{array}
\right\}.
\end{equation}
By replacing the set $\U'$ with ${\U}^{**}$ in \eqref{equ:pldr}, one can obtain a linear decision rule approximation problem with a polyhedral uncertainty set, which can be reformulated to a tractable linear program if the recourse matrix $\BB'\left(\vv\right)$ is fixed. We denote by $Z^{\mathcal{PL}}_{\textup{GWK}}$ its optimal value.  We now examine the relation between $Z^{\mathcal{PL}}_{\textup{IA}}$ and $Z^{\mathcal{PL}}_{\textup{GWK}}$. To this end,  by using the copositive programming techniques, we first propose a looser outer approximation $\U^*$ of the lifted set $\U'$  and establish that the set is still tighter than $\U^{**}$. We define this outer approximation as 
\begin{equation*}
\label{eq:expanded_Georghiou_}
\U^* = \left\{
(\bm w, \uu) \in \RR^{KL} \times  \mathcal U :
\begin{array}{ll}
\bm z \in \RR_+^{K(L+1)} \\
\bm Z^{k,\ell}\in\mathcal S^3,\; \bm Z^{k,\ell}\succeq\bm 0 & \forall \, k \in[K] \  \ell \in [L]  \\  
w_{k,\ell} = z_{k,\ell} - z_{k,\ell+1}& \forall \, k\in[K]\  \ell \in [L]\\
z_{k,1} = u_{k}-\underline u_{k},\; z_{k,L+1}=0 & \forall \, k \in [K]\\
z_{k,\ell} \geq u_k-h_{k,\ell} ,\   \overline{u}_k \geq z_{k,\ell} & \forall \, k \in [K] \  \ell \in [L+1]  \\
Z^k_{1,3}-Z^k_{3,3}+ z_{k,\ell+1}(h_{k,\ell-1}-h_{k,\ell+1}) \geq 0 & \forall \, k \in [K]\  \ell \in [L] \\[1mm]
Z^k_{2,3}-Z^k_{3,3}+ z_{k,\ell+1}(h_{k,\ell}-h_{k,\ell+1}) \geq 0& \forall \, k\in[K]\  \ell \in [L] \\[1mm]
Z^k_{1,3}-Z^k_{1,1}+z_{k,\ell-1}(h_{k,\ell+1}-h_{k,\ell-1}) \\
\qquad+Z^k_{2,2}-Z^k_{2,3}+z_{k,\ell}(h_{k,\ell}-h_{k,\ell+1})\geq 0 & \forall \, k \in [K]\  \ell \in [L] \\[1mm]
Z^k_{3,3}-Z^k_{1,3}+z_{k,\ell+1}( h_{k,\ell+1}-h_{k,\ell-1}) \\
\qquad+Z^k_{1,2}-Z^k_{2,2}+z_{k,\ell}(h_{k,\ell-1}-h_{k,\ell})\geq 0& \forall \, k \in [K] \  \ell \in [L] \\[1mm]
Z^k_{2,3}-Z^k_{2,2}+z_{k,\ell}(h_{k,\ell+1}-h_{k,\ell}) \geq 0 & \forall \, k \in [K] \  \ell \in [L] \\[1mm]
\end{array}
\right\},
\end{equation*}
which is a concise set involving $\mathcal O(KL)$ semidefinite constraints  of size $3\times 3$. 
The following proposition shows the chain relation of $\U', \, \U^*$, and $\U^{**}$.
\begin{proposition} \label{prop:our_vs_Georghiou}
	We have $\U' \subseteq \U^*\subseteq \U^{**}$. 
\end{proposition}
\begin{proof}
See the Appendix.  
\end{proof}

Finally, we are ready to state our main result of this section in the following theorem.  
\begin{theorem}
	\label{thm:our_vs_Georghiou}
 	We have $Z^{\mathcal{PL}}_{\textup{IA}}\leq Z^{\mathcal{PL}}_{\textup{GWK}}$. 
\end{theorem}
\begin{proof}
See the Appendix.  
\end{proof}
\noindent The proof of Theorem \ref{prop:our_vs_Georghiou} imparts the favorable  insight that a tighter approximation can already be obtained by considering a  concise set involving $\mathcal O(KL)$ semidefinite constraints  of size $3\times 3$.

\section{Copositive reformulation for multi-stage decision rule problems}
\label{sec:multi-stage}
We now extend the proposed copositive programming approach to 
multi-stage robust optimization problems
 of the following generic form:
\begin{equation} \label{equ:msro}
\begin{array}{cl} 
\inf  & \cc^\top \bm x +  \sup \limits_{\bm{u} \in \U} \ \sum \limits_{t = 1}^T \bm d_t(\uu^t)^\top \bm y_t(\uu^t) \\
\st & \A(\uu^1) \bm x + \sum\limits_{t=1}^T \BB_t(\uu^t)\bm y_t(\uu^t) \geq \bm h(\uu) \ \ \ \forall \, \uu \in \U \\
     & \bm x \in \X, \ \bm y_t \in \FF_{K^t + 1, \, N_t} \ \forall \, t \in [T]. 
\end{array}
\end{equation}
The vector~$\bm u^t$ in \eqref{equ:msro} collects the history of observations up to time~$t$, and is defined as
\begin{equation*}
\label{equ:ut_history}
\bm u^t= (\uu_1, \ldots, \uu_t, 1)\in\RR^{K^t + 1},
\end{equation*}
where  $\bm u_t\in \RR^{K_t}$ contains uncertain parameters observed at time $t\in[T]$, and $K^t := \sum_{s=1}^t K_{s}$. Here, we have appended the constant scalar $1$ at the end of the vector so that affine functions in $(\bm{u}_1, \ldots, \bm{u}_t)$ can be represented  as linear functions in $\bm u^t$, while quadratic functions  in $(\bm{u}_1, \ldots, \bm{u}_t)$ can be formulated compactly in a homogenized manner. We set the vector of all uncertain  parameters in \eqref{equ:msro} to $\bm u:=\bm u^T\in\RR^{K+1}$, with $K=K^T$. As in the two-stage setting, the problem parameters $\A(\uu^1), \BB_t(\uu^t)$, $\bm d_t(\uu^t)$ and $\bm h(\uu)$
are described by linear functions in their respective arguments, as follows,
\[
\A(\uu^1) := \sum\limits_{k=1}^{K^1+1} [\uu^1]_{k}   \bm{\widehat A}_{k}, \ \  \BB_t(\uu^t) :=  \sum\limits_{k=1}^{K^t+1} [\uu^t]_{k} \, \bm{\widehat B}_{k,t}, \ \
\bm d_t(\uu^t) := \bm{\widehat D}_t\uu^t, \  \ \bm h(\uu) := \bm{\widehat H}\uu,
\]  
where $\bm{\widehat A}_{k} \in \RR^{J \times M},  \ \bm{\widehat B}_{k,t} \in \RR^{J \times N_t}$,
$\bm{\widehat D}_t := (\bm{\widehat d}_{1,t}, \ldots, \bm{\widehat d}_{N_t,t})^\top \in \RR^{N_t \times (K^t+1)}$,
and $\bm{\widehat H} := (\bm{\widehat h}_{1}, \ldots, \bm{\widehat h}_{J})^\top \in \RR^{J \times (K+1)}$ are deterministic data.

The  decision vector $\bm y_t(\bm u^t)\in \RR^{N_t}$ in \eqref{equ:msro} is chosen 
after the realization of uncertain parameters up to time~$t$ but before the revelation 
of future outcomes $\{ \uu_{s} \}_{ s \in [t+1, T]}$. The objective of problem~\eqref{equ:msro} is
to find a here-and-now decision $\bm x \in \X$ and a sequence of nonanticipative decision
rules $\{\bm y_t(\cdot)\}_{t\in[T]}$ that are feasible to the semi-infinite constraint in~\eqref{equ:msro} 
and  minimize the total cost $\bm c^\top \bm x + \sup_{\uu \in \U} \, 
\sum_{t=1}^T \bm d_t(\uu^t)^\top \bm y_t(\uu^t)$. Problem \eqref{equ:msro} constitutes an extension of the  
two-stage problem~\eqref{equ:tsro} to the multi-stage setting, and as such is computationally challenging to 
solve. To this end, we endeavor to derive copositive programming reformulations in view of linear and quadratic 
decision rules. Tractable semidefinite programming approximations can then be derived using the techniques 
discussed in Section \ref{sec:SDP_approx}. One can further enhance these approximations by utilizing piecewise 
linear and piecewise quadratic decision rules discussed in Section \ref{sec:enhanced}. 

As in the two-stage setting, we assume that the uncertainty set $\U$ is defined as in (\ref{equ:uncertainty})
and satisfies both Assumptions~\ref{ass:bounded} and~\ref{ass:quadcons}. In the following, we use the 
linear truncation operator $\bm \Pi_t : \RR^{K+1} \mapsto \RR^{K^t+1}$ that satisfies
\[
\bm \Pi_t\uu=\uu^t \ \ \ \forall \, \uu \in \RR^{K+1}.
\]
We first examine the case when the multi-stage robust optimization problem has non-fixed
recourse. Here, we apply the linear decision rules
\[
\bm y_t(\uu^t) = \bm Y_t \uu^t =  \bm Y_t\bm \Pi_t \uu,
\]
for some coefficient matrix $\bm Y_t \in \RR^{N_t \times (K^t+1)}$. This gives rise to the following conservative approximation
of problem (\ref{equ:msro}):
\begin{equation} \label{equ:msro-ldr}
\tag{$\mathcal{ML}$}
\begin{array}{rcl} 
Z^{\mathcal{ML}}=&\inf  & \cc^\top \bm x +  \sup \limits_{\bm{u} \in \U} \ \sum \limits_{t = 1}^T  \bm d_t(\uu^t)^\top  \bm Y_t\bm \Pi_t \uu\\ 
&\st & \A(\uu^1) \bm x + \sum\limits_{t=1}^T \BB_t(\uu^t)\bm Y_t \bm\Pi_t \uu \geq \bm h(\uu) \ \ \ \forall \, \uu \in \U \\
&     & \bm x \in \X, \ \bm Y_t \in \RR^{N_t \times (K^t+1)} \ \forall \, t \in [T].
\end{array}
\end{equation} 
Problem \eqref{equ:msro-ldr} shares the same structure as its two stage counterpart \eqref{equ:ldr}. 
Hence, by employing the same reformulation techniques described in Section \ref{sec:ldr}, we can 
derive a polynomial size  copositive program for the problem. 
\noindent For notational convenience, in the following we define the matrices
{\small
\begin{equation*} \label{equ:ABmat_T1}
\begin{array}{l}
\bm{\widehat \Theta}_{j} := \begin{pmatrix} \mathbf{e}_{j}^\top\bm{\widehat A}_{1} \\ \vdots \\ \mathbf{e}_{j}^\top\bm{\widehat A}_{K^1+1} \end{pmatrix}  \in \RR^{(K^1+1) \times M},\quad
\bm{\widehat \Lambda}_{j,t} := \begin{pmatrix} \mathbf{e}_{j}^\top\bm{\widehat B}_{1,t} \\ \vdots \\ \mathbf{e}_{j}^\top\bm{\widehat B}_{K^t+1,t} \end{pmatrix}  \in \RR^{(K^t+1) \times N_t} \qquad \forall \, t \in [T] \ \forall \, j \in [J],
\end{array}
\end{equation*}
}
\noindent  and the affine functions
\begin{equation*} \label{equ:OmegaxY1}
\begin{array}{l}
\bm{\Omega}_{j} \left( \bm x, \,  \bm Y_1, \ldots, \bm Y_T\right) := 
  \dfrac12 \bm \Pi_1^\top ( \bm{\widehat \Theta}_{j} \x \mathbf{e}_{K^1+1}^\top + \mathbf{e}_{K^1+1}\x^\top \bm{\widehat \Theta}_{j}^\top) \bm \Pi_1 \\
 \qquad\qquad+ \dfrac12 \sum\limits_{t=1}^T \bm \Pi_t^\top(  \bm{\widehat \Lambda}_{j,t}\bm Y_t + \bm Y_t^\top\bm{\widehat \Lambda}_{j,t} ) \bm \Pi_t - \dfrac12(\bm{\widehat h}_j\mathbf{e}_{K+1}^\top + \mathbf{e}_{K+1}\bm{\widehat h}_j^\top) \quad\forall j\in[J].
\end{array}
\end{equation*}
The equivalent reformulation is provided in the following theorem. We omit the proof as it closely follows that of Theorem \ref{the:ldr-ref}. 
\begin{theorem} \label{the:msro-ldr-ref}
Problem 
(\ref{equ:msro-ldr}) is equivalent to the following copositive program:
\begin{equation} \label{equ:cop_msro_ldr}
\begin{array}{rcll} 
Z^{\mathcal{ML}}=&\inf  & \cc^{\top} \x + \lambda \\
&\st  & \lambda \, \mathbf{e}_{K+1}\mathbf{e}_{K+1}^\top -  \dfrac12 \sum\limits_{t=1}^T  \bm \Pi_t^\top\left( \mathbf{\widehat D}_t^\top\bm{Y}_t + \bm{Y}_t^\top \mathbf{\widehat D}_t \right)\bm \Pi_t  + \sum\limits_{i=1}^I \alpha_i \bm{\widehat C}_i \in \COP(\K) \\
&     &  \bm{\Omega}_{j} \left( \bm x, \,  \bm Y_1, \ldots, \bm Y_T\right)  - \pi_{j} \, \mathbf{e}_{K+1}\mathbf{e}_{K+1}^\top - \sum\limits_{i=1}^I  \mathbf{e}_i^\top \bm{\beta}_{j} \bm{\widehat C}_i  \in \COP(\K) \ \ \ \forall \, j \in [J]  \\   
&     & \lambda \in \RR, \ \x \in \mathcal X,  \ \bm \alpha \in \RR^I, \  \bm \pi \in \RR_+^{J},  \  \bm \beta_{j} \in \RR^I  \ \ \forall \,  j \in [J], \ \bm{Y}_t \in \RR^{N_t \times (K^t+1)} \  \ \forall \, t \in [T].
\end{array}
\end{equation}
\end{theorem}

Next, we consider the case when the multi-stage problem has  fixed recourse, 
i.e.,
\[
\begin{array}{l}
\dd_t(\uu^t)=\bm{\widehat d}_t   \  \ \ \text{and} \ \  \  \BB_t(\uu^t)=\bm{\widehat B}_t \qquad  \forall \, \uu^t \in \RR^{K^t+1} \ \forall \, t \in [T], 
\end{array}
\]
where 
$\bm{\widehat d}_t \in \RR^{N_t}$ and  $\bm{\widehat B} \in \RR^{J \times N_t}$ 
are deterministic vector and matrix, respectively. 
Here, we can apply the quadratic decision rules
\[
[\bm y(\uu^t)]_{n_t} = (\uu^t)^\top \bm Q_{n_t,t} \uu^t=(\bm\Pi^t\uu)^\top \bm Q_{n_t,t} \bm\Pi^t\uu \qquad \forall \, n_t \in [N_t],
\]
for some coefficient matrices $\bm Q_{n_t,t} \in \SYM^{K^t+1}$,$n_t \in [N_t]$, $t \in [T]$. 
This yields the following conservative approximation of problem \eqref{equ:msro}:
\begin{equation} \label{equ:msro-qdr-ref}
\tag{$\mathcal{MQ}$}
\begin{array}{rcll} 
Z^{\mathcal{MQ}}=&\inf  &\displaystyle \cc^{\top}\x +\sup\limits_{\uu \in \U} \sum \limits_{t=1}^T \sum\limits_{n_t=1}^{N_t} \widehat d_{n_t, t} (\bm\Pi^t\uu)^\top \bm Q_{n_t,t} \bm\Pi_t\uu\\
&\st  &  \displaystyle (\bm\Pi_1\uu)^\top\bm{\widehat \Theta}_j \bm x + \sum\limits_{t=1}^T \sum\limits_{n_t=1}^{N_t}\left( b_{j,n_t} (\bm\Pi_t\uu)^\top \bm Q_{n_t,t} \bm\Pi_t\uu\right)\geq \bm h(\uu) \ \ \ \forall \, \uu \in \U \\
&& \x \in \X,\;\bm Q_{n_t, t} \in \SYM^{K^t+1} \ \ \ \forall \, t \in [T] \ \forall \,  n_t \in [N_t]. 
\end{array}
\end{equation}
\noindent Problem \eqref{equ:msro-qdr-ref} shares the same structure as its two-stage counterpart \eqref{equ:qdr-ref}, which indicates that it is also amenable to an equivalent copositive programming reformulation. To this end, we define the affine functions
\begin{equation*} \label{equ:GxQ1}
\begin{array}{l}
\bm{\Gamma}_{j}\left( \bm x, \bm Q_{1, 1} \ldots, \bm Q_{N_T, T} \right) := \dfrac12 \bm{\Pi}_1^\top( \bm{\widehat \Theta}_{j} \x \mathbf{e}_{K^1+1}^\top
 + \mathbf{e}_{K^1+1}\x^\top \bm{\widehat \Theta}_{j}^\top)\bm{\Pi}_1\\
 \qquad\qquad\qquad - \dfrac12(\mathbf{e}_{K+1}\bm{\widehat h}_{j}^\top - \bm{\widehat h}_{j} \mathbf{e}_{K+1}^\top) + \sum\limits_{t=1}^T\sum \limits_{n_t=1}^{N_t} \, \widehat b_{j,n_t} \bm \Pi_t^\top \bm{Q}_{n_t,t}  \bm{\Pi}_t\qquad\forall j\in[J]. 
\end{array}
\end{equation*}
The equivalent reformulation is provided in the following theorem whose proof is omitted as it closely follows that of Theorem \ref{thm:qp}. 
\begin{theorem}
Problem (\ref{equ:msro-qdr-ref}) is equivalent to the following copositive program:
\begin{equation} \label{equ:cop_msro_qdr}
\begin{array}{rll}
Z^{\mathcal{MQ}}=&\min  & \cc^\top \x + \lambda \\
&\st & \lambda \, \mathbf{e}_{K+1}\mathbf{e}_{K+1}^\top - \sum\limits_{t=1}^T \sum \limits_{n_t=1}^{N_t} \,\left[ \widehat d_{n_t,t} \bm \Pi_t^\top \bm{Q}_{n_t,t} \bm \Pi_t \right]+ \sum\limits_{i=1}^I  \alpha_i \bm{\widehat C}_i \in \COP(\K) \\
&    & \bm{\Gamma}_{j}\left( \bm x, \bm Q_{1, 1} \ldots, \bm Q_{N_T, T} \right)  - \pi_j \mathbf{e}_{K+1}\mathbf{e}_{K+1}^\top - \sum\limits_{i=1}^I  [\mathbf{e}_i^\top \bm \beta_{j}] \,  \bm{\widehat C}_i  \in \COP(\K)  \ \ \ \forall \, j \in [J] \\
&    & \lambda \in \RR, \ \x \in \X, \  \bm \alpha \in \RR^I, \ \bm{\pi} \in\RR_+^{J}, \ \bm \beta_{j} \in \RR^I \ \ \ \forall \,  j \in [J] \\
&    &  \bm Q_{n_t,t}\in\SYM^{K^t+1} \ \ \ \forall \, t \in[T] \, \  \forall \, n_t \in [N_t]  
\end{array}
\end{equation}
\end{theorem}
\begin{remark}
In some multi-stage robust optimization problems, we may observe that some of the 
recourse decision variables are multiplied with uncertain parameters, while the 
remaining recourse decisions are multiplied with deterministic terms. In such situations,
we can apply quadratic decision rules to the latter, which gives rise to
stronger decision rule approximations. With minimum modification
we can reformulate the decision rule problem into an equivalent copositive program similar to~\eqref{equ:cop_msro_qdr}. 
We omit the detailed reformulation here.  
\end{remark}

\section{Numerical experiments}
\label{sec:numerical}

In this section, we assess the effectiveness of our copositive programming approach over three applications in operations management. The first example is a multi-item newsvendor problem, which can be reformulated to a two-stage robust optimization problem with fixed recourse. The following two examples are inventory control and index tracking problems, which correspond to multi-stage robust optimization problems with non-fixed recourse. 
All optimization problems are solved using MOSEK 8.1.0.56~\cite{Mosek} via the YALMIP interface~\cite{Lofberg2004yalmip}~on a 16-core 3.4 GHz  Linux PC with 32 GB RAM. 

\subsection{Multi-item newsvendor}
We consider the following robust multi-item newsvendor problem studied 
in \cite{Ardestani-Jaafari.Delage.2016}:
\begin{equation} \label{equ:vendor}
\max \limits_{\bm x \ge \bm 0} \min \limits_{\bm \xi \in \Xi} \sum_{n =1}^N
\big(
    r_n \min(x_n, \xi_n) -c_n x_n - s_n \max (\xi_n - x_n, 0)
\big). 
\end{equation}
Here, $N$ represents the number of products; $\bm x$ is the vector of order quantities; $\bm \xi$ is the vector of uncertain demands; 
$\bm r$, $\bm c$, and $\bm s$ are the vector of sales prices, order costs, and shortage costs, respectively. We assume that the products do not have a salvage value, and the salvage value is set to~$0$. 
Problem (\ref{equ:vendor}) can be reformulated as the  two-stage robust optimization problem given by
\begin{equation} \label{equ:two-stage-vendor}
\begin{array}{lll}
\max \limits_{\bm x, \bm y(\cdot)} & \min \limits_{\bm \xi \in \Xi} \sum\limits_{n=1}^N y_n(\bm \xi) & \\
\st & y_n(\bm \xi) \le (r_n - c_n) x_n - r_n (x_n - \xi_n) &  \bm \xi \in \Xi, \,\forall \, n \in [N] \\  
     & y_n(\bm \xi) \le (r_n - c_n)x_n - s_n(\xi_n - x_n) &  \bm \xi \in \Xi, \, \forall \, n \in [N] \\
     & \bm x \ge \bm 0. 
\end{array}
\end{equation}
In this problem, the uncertainty set  is specified through a factor model defined as
\[
\Xi := \left\{ \bm \xi \in \RR^N : 
\begin{array}{l}
\bm \xi = \bm{\bar \xi} + \Diag(\bm{\hat \xi}) \bm{F}\bm{\zeta},  \\
\bm \zeta \in \RR^N ,\;\|\bm \zeta\|_\infty \leq 1, \, \|\bm \zeta\|_1 \leq \rho 
\end{array}
\right\},
\]
where $\bm \zeta$ is a vector comprising all factors, $\bm F \in \RR^{N \times N}$ is the factor loading matrix, and $\rho < N$ 
is a scalar that controls the level of conservativeness. 
The associated cone $\K$ related to this uncertainty set is written as
\[
\K := \left\{ (\bm \xi, \tau) \in \RR^N \times \RR_+ : 
\begin{array}{l}
\bm \xi = \bm{\bar \xi}\tau + \Diag(\bm{\hat \xi}) \bm{F}\bm{\zeta},  \\
\bm \zeta \in \RR^N ,\;\|\bm \zeta\|_\infty \leq \tau, \, \|\bm \zeta\|_1 \leq \rho\tau
\end{array}
\right\}.
\]

As the problem has fixed recourse, we can apply the quadratic decision rule scheme (QDR) proposed in Section~\ref{sec:qdr}  and solve the semidefinite approximation which results from replacing the copositive cone $\COP(\K)$  with the inner approximation $\IA(\K)$ defined in (\ref{equ:innercone}). We compare our QDR scheme with the one proposed by Ben-Tal et al.~(BGGN)~where we replace the cone $\COP(\K)$  with the inner approximation $\AS(\K)$ defined in (\ref{equ:ascone}), with the polynomial decision rule scheme of degree~$3$ (PDR3), and with the piecewise linear decision rule scheme proposed by Georghiou et al.~(GWK)~\cite{Georghiou.Wiesemann.Kuhn.2015}. In addition, we also compare our method with state-of-the-art schemes for two-stage robust optimization problems with fixed recourse:  the method (COP) described in~\cite{Xu.Burer.2018} and the method (AJD) described in~\cite{Ardestani-Jaafari.Delage.2016}. 
We note that these two methods generate the same solutions with comparable computational times.

All experimental results are averaged over $100$ random instances. We utilize the mechanism in~\cite{Ardestani-Jaafari.Delage.2016} to set up the parameters and to generate the random instances. For each instance, we consider $N=5$ items, and set $\bm r=80\mathbf e$ and $\bm p=60\mathbf e$. We further sample the vector $\bm c$ uniformly at random from the hypercube $[40,60]^5$. For the uncertainty set, we set $\rho = 4$ and $\bm{\bar \xi} = 60\mathbf e$, while the vector  $\bm{\hat \xi}$ 
is generated uniformly at random  from $[50,60]^5$.  We sample each entry of the matrix $\bm F$ uniformly 
from $[-1,1]$, and normalize each row so that its sum is equal to $1$. 
Table \ref{tab:newsvendor_value} reports several statistics of  relative gaps between the optimal value of QDR and those of the other alternative methods. 
We find that QDR provides a substantial average improvement of $52\%$ over BGGN and an average improvement of $22.3\%$ over GWK. Rather surprisingly, we also find that QDR outperforms the state-of the-art COP and AJD schemes by 6\%. Table \ref{tab:newsvendor_value} indicates that
QDR generates the same performance as the  less tractable PDR3. 
Table \ref{tab:newsvendor_time} reports the average computation times of the four methods. We observe that QDR can be solved as fast as BGGN, GWK, COP, and AJD, while it takes 40 times as long to solve PDR3.    
In summary, we may thus conclude that QDR provides high-quality solutions in a very efficient manner. 
{\color{red}
\begin{table}[tbp]
\centering
\begin{tabular}{c|cccccc}  
\multicolumn{1}{c}{} & \multicolumn{5}{c}{Approximation method} \\ \hline
Statistic &  BGGN & GWK& COP  & AJD & PDR3    \\ \hline
 $10$th percentile & 26.5& 26.5 & 2.3 & 2.3  & 0  \\ 
  Mean & 52.0 & 52.0 &  6.0 & 6.0 & 0 &  \\ 
  $90$th percentile & 87.3& 87.3 &  9.7 & 9.7 & 0  \\ 
\hline\hline
\end{tabular}
\caption{Relative gaps (in percent) between the alternative approximation schemes and QDR}
\label{tab:newsvendor_value}
\end{table}
\begin{table}[tbp]
\centering
\begin{tabular}{c|cccccc}  
 & BGGN  & GWK & COP & AJD &  QDR & PDR3    \\ \hline
Time & 1.68 & 1.75 & 1.61 & 1.59 & 1.62 & 62.17   \\
\hline\hline
\end{tabular}
\caption{The average computation times (in seconds) of the different approximation schemes}
\label{tab:newsvendor_time}
\end{table}
}
\begin{remark} Since  COP corresponds to a semidefinite programming approximation of the exact copositive reformulation of the newsvendor problem, it is indeed very surprising that QDR can outperform COP. 
For the  temporal network example  described  in~\cite{Xu.Burer.2018} where the uncertainty set  is given by a $1$-norm ball, 
one can formally prove that  QDR performs better than COP. In general, however, we cannot prove that one 
approximation is tighter than the other, or vice versa.
\end{remark}

We also assess the quality of the first-stage decisions (order quantities) obtained from the different approximation methods by evaluating their true worst-case profits. Since the profit function in \eqref{equ:vendor} is concave, the worst-case profit of any fixed decision occurs at a demand scenario from an extreme point of the uncertainty set $\Xi$. Thus, we can enumerate all extreme points of the uncertainty set\footnote{In general, it is computationally prohibitive to enumerate the extreme points of a polyhedral set. However, it is manageable for our case since there are only a few variables and constraints involved.} and find the one that minimizes the profit to determine the worst-case scenario profit of each first-stage decision. Table~\ref{tab:newsvendor_simulation} reports the statistics of relative gaps between the worst-case-scenario profit of our method and those of other methods. We find that the proposed QDR scheme provides substantial average improvements of 36.3\%, 36.3\%, 4.7\%, and 4.7\% over BGGN, GWK, COP, and AJD, respectively.

\begin{table}[tbp]
\centering
\begin{tabular}{c|ccccccc}  
\multicolumn{1}{c}{} & \multicolumn{5}{c}{Approximation method} \\ \hline
Statistic &  BGGN & GWK & COP  & AJD & PDR3    \\ \hline
 $10$th percentile & 14.5& 14.5 & 2.1& 2.1  & 0  \\ 
  Mean & 36.3 & 36.3 &  4.7 & 4.7 & 0 &  \\ 
  $90$th percentile & 76.2& 76.2 &  9.3 & 9.3 & 0  \\ 
\hline\hline
\end{tabular}
\caption{Relative gaps (in percent) of the worst-case-scenario profits between the alternative approximation schemes and QDR}
\label{tab:newsvendor_simulation}
\end{table}

Finally, we analyze the optimal decision rules from the different approximation methods by considering an instance of robust two-item newsvendor problems.  
Figure~\ref{fig:newsvendor} visualizes the decision rules from the different methods as a function of the demand $\xi_2$ of the second item. We observe that the quadratic function  from QDR and the polynomial function from PDR3  coincide with the optimal decision rules (Optimal) at the extreme points of the uncertainty set.
This  implies that QDR and PDR3 can generate optimal order quantities  as they anticipate for the worst-case demand scenarios. On the other hand, BGGN generates suboptimal order quantities as its decision rule function  does not coincide with the extremes of the optimal decision rules. We note that the COP method does not generate any decision rules.
We further remark that GWK and BGGN return the same decision rule, and thus we only plot the one from BGGN. 

\begin{figure}[ht]
  \centering
  \includegraphics[width=4.5in]{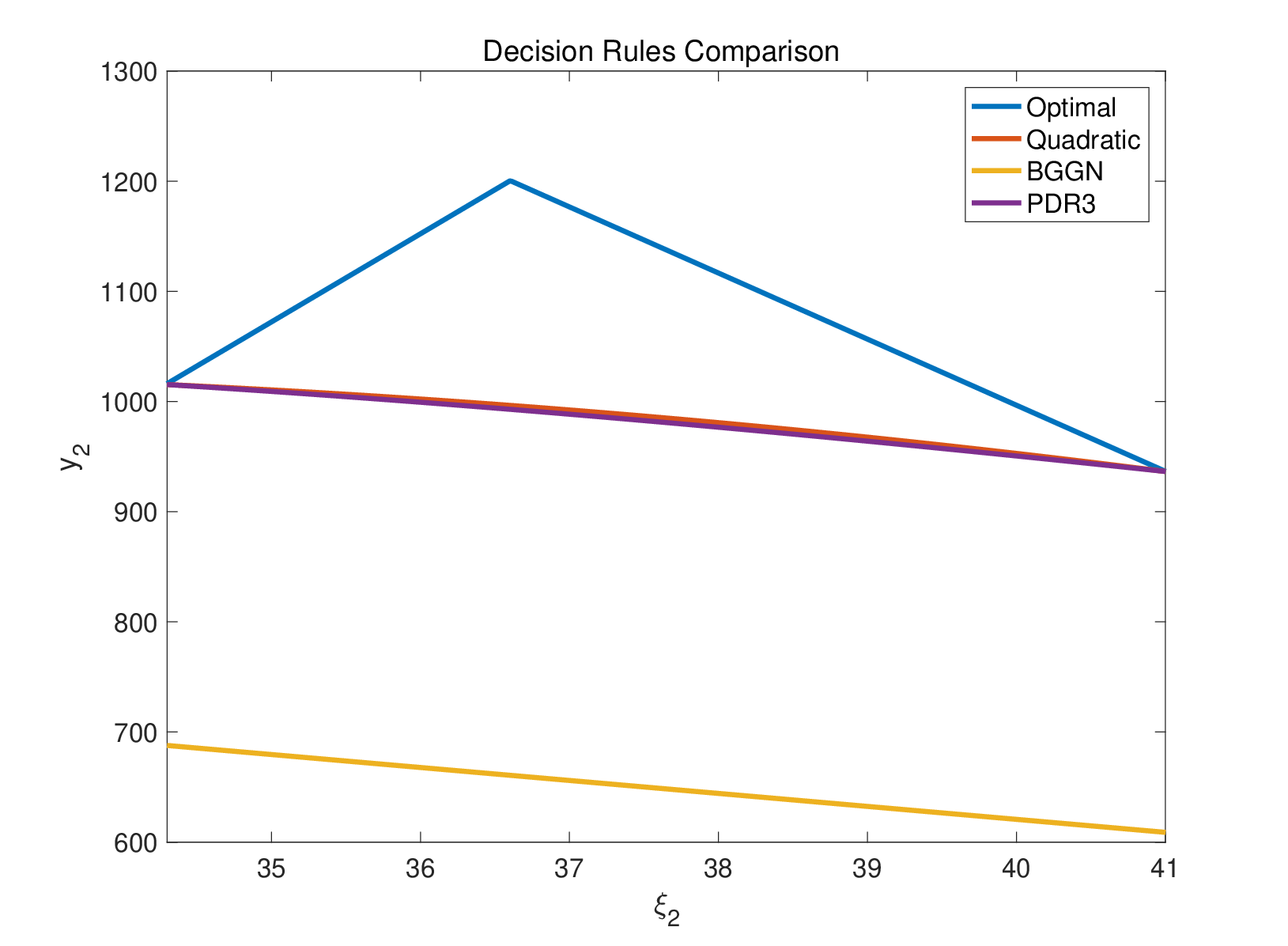}
  \caption{Comparison of the decision rule functions from QDR, BGGN, PDR3, and GWK.}
\label{fig:newsvendor}
\end{figure}

\subsection{Inventory control}

We next consider a multi-stage robust inventory control problem with multiple products 
and backlogging. A stochastic programming version of the problem is described 
in~\cite{Georghiou.Wiesemann.Kuhn.2015}.
In this problem, we must determine sales
and order policies that maximize the worst-case profit over a planning horizon of $T$ time stages. 
At the beginning of each time stage $t$, we observe a vector of risk factors~$\bm \xi_t$ that 
explains the uncertainty in the current demand $D_{t,p}(\bm \xi_t)$ and the unit sales price 
$R_{t,p}(\bm \xi_t)$ of each product $p \in[P]$. After $\bm \xi_t$ is revealed at time 
stage $t$, we must determine the quantity $s_{t,p}$ of product $p$ to sell at the current price, 
the amount $o_{t,p}$ of product $p$ to replenish the inventory, and the amount $b_{t,p}$ of
product $p$ to backlog to the next time stage at the unit cost $C_b$. 
The sales~$s_{t,p}$ of product $p$ at time stage $t$ can only be provided by orders placed at 
time stage $t-1$ or earlier. We denote the inventory level at the beginning of each time stage
$t$ by $I_t$. For simplicity, we assume that one unit of each product occupies the same 
amount of space and incurs periodically the same inventory holding costs $C_h$. 
The inventory level is required to remain nonnegative and is not allowed to exceed the capacity 
limit $\bar I$ throughout the planning time horizon. The inventory control problem can be 
stated as the MSRO problem
\begin{equation} \label{equ:inventory}
\begin{array}{ll}
\max &  \min\limits_{\bm \xi \in \Xi} \, \sum\limits_{t=1}^T \sum\limits_{p=1}^P \left[ R_{t,p}(\bm \xi_t)s_{t,p}(\bm \xi^t) - C_bb_{t,p}(\bm \xi^t) - C_h I_{t,p}(\bm \xi^t) \right] \\
\st    & I_{1,p}(\bm \xi^1) = I_{0, p}  - s_{1,p}(\bm \xi^1), \  b_{1,p}(\bm \xi^1) = D_{1,p}(\bm \xi_1) - s_{1,p}(\bm \xi^1) \ \ \  \forall \, \bm \xi \in \Xi, \ \forall \, p \in [P]  \\
        & I_{t,p}(\bm \xi^t) = I_{t-1,p}(\bm \xi^{t-1}) + o_{t,p}(\bm \xi^{t-1}) - s_{t,p}(\bm \xi^t) \ \ \  \forall \, \bm \xi \in \Xi, \, \forall \, p \in [P], \, \forall \, t \in [T]\backslash \{1\}   \\
        & b_{t,p}(\bm \xi^t) = b_{t-1,p}(\bm \xi^{t-1}) + D_{t,p}(\bm \xi_t) - s_{t,p}(\bm \xi^t) \ \ \ \forall \, \bm \xi \in \Xi, \, \forall \, p \in [P],  \, \forall \, t \in [T]\backslash \{1\}  \\
        & o_{t,p}(\bm \xi^t), s_{t,p}(\bm \xi^t), b_{t,p}(\bm \xi^t), I_{t,p}(\bm \xi^t) \geq 0, I_{t,p}(\bm \xi^t) \leq \bar I \ \ \  \forall \, \bm \xi \in \Xi, \, \forall \, p \in [P], \, \forall \,  t \in [T],
\end{array}
\end{equation} 
where $I_{0,p}$ are fixed to pre-specified quantities for all $p \in [P]$. The product prices are defined as 
\[
R_{t,p}(\bm \xi_t) = 4 + \alpha_{1,p}\xi_{t,1} + \alpha_{2,p}\xi_{t,2} + \alpha_{3,p}\xi_{t,3} + \alpha_{4,p} \xi_{t,4} 
\]
with factor loadings $\alpha_{1,p}, \alpha_{2,p}, \alpha_{3,p}, \alpha_{4,p} \in [-1,1]$. Similarly, we set the demands to 
\[
D_{t,p}(\bm \xi_t) = 2 + \sin(\frac{2\pi(t-1)}{12}) + \frac12[ \beta_{1,p}\xi_{t,1} + \beta_{2,p}\xi_{t,2} + \beta_{3,p}\xi_{t,3} + \beta_{4,p} \xi_{t,4} ]
\]
for $p = 1, \ldots, P/2$ and
\[
D_{t,p}(\bm \xi_t) = 2 + \cos(\frac{2\pi(t-1)}{12}) + \frac12[ \beta_{1,p}\xi_{t,1} + \beta_{2,p}\xi_{t,2} + \beta_{3,p}\xi_{t,3} + \beta_{4,p} \xi_{t,4} ]
\]
for $p = 1/P + 1, \ldots, P$ with factor loadings $\beta_{1,p}, \beta_{2,p}, \beta_{3,p}, \beta_{4,p} \in [-1,1]$. 
The sine (cosine) terms in the above expression correspond to the stylized fact that
the expected demands of the first (last) P/2 products are high in spring (winter) and low in fall (summer). 
We assume that the vectors of risk factors $\bm \xi_t \in \RR^4$ for all $t = 1, \ldots, T$, are serially
independent and uniformly distributed on~$[-1,1]^4$. Formally, the uncertainty set is defined as
\[
\Xi := \left\{ (\bm \xi := \bm \xi_1, \ldots, \bm \xi_t) : \|\bm \xi_t\|_{\infty} \leq 1 \ \forall \, t \in [T]\right \}. 
\] 
The associated cone $\K$ is written as follows:
\[
\K := \left\{ (\bm \xi, \tau) \in \RR^{4T}\times \RR_+ : \|\bm \xi_t\|_{\infty} \leq \tau \ \forall \, t \in [T]\right \}. 
\] 

In all numerical experiments, we generate $25$ random instances of the inventory control problem 
with $P = 4$ products. We  utilize the  mechanism in~\cite{Georghiou.Wiesemann.Kuhn.2015} to set up the parameters and to generate the random instances.
We  set backlogging and inventory holding costs  identically to 
$C_b = C_h = 0.2$. We further set the initial inventory level to $I_{0,p} = 0$ and the inventory capacity 
to $\bar I = 24$. We sample the factor loadings $\alpha_{1,p}, \alpha_{2,p}, \alpha_{3,p}, \alpha_{4,p}$
and $\beta_{1,p}, \beta_{2,p}, \beta_{3,p}, \beta_{4,p}$ uniformly from the interval~$[-1,1]$.
As problem \eqref{equ:inventory} has non-fixed recourse, we  employ  linear decision rules, and further enhance them by applying the 
piecewise scheme discussed in Section~\ref{sec:enhanced}, where the folding directions are described by the standard basis vectors 
$\mathbf e_\ell$, $\ell\in[4]$. This gives  rise to  a semidefinite approximation 
which results from replacing the copositive cone $\COP(\K)$ in the equivalent copositive program with the inner approximation 
$\IA(\K)$ defined in (\ref{equ:innercone}). We compare our scheme~(PLDR) with  the one proposed by Ben-Tal et al.~(BGGN)~where 
we replace the cone $\COP(\K)$  with the inner approximation $\AS(\K)$ defined in~(\ref{equ:ascone}), 
and with polynomial decision rule scheme of degree~$3$ (PDR3). 

We test the different schemes on  problem instances with planning horizons $T = 1$, $3$,  $6$,  $9$, $12$, $15$, $18$,  $21$, and $24$. 
 Table \ref{tab:inventory_value} reports the relative gaps between the 
optimal values of PLDR and those of the other two schemes, while Table \ref{tab:inventory_time} shows the average computation times for 
the three approximation schemes. Note that PDR3 can only solve instances up to  $T = 3$ before it starts experiencing  numerical  issues. 
As illustrated in~Table \ref{tab:inventory_value}, 
the relative gap between PLDR
and BGGN increases dramatically with the planning horizon, where the largest average improvement of $191.2\%$ is observed for $T=24$. 
Meanwhile, PLDR can generate the same results as PDR3 
in the case of $T=1$, and remain very close to PDR3 for
$T=3$.   As illustrated in  these tables, our proposed copositive scheme
can return solutions that are of very high quality without sacrificing much computational effort.  

\begin{table}[tbp]
\centering
\begin{tabular}{c|c|ccccccccc}  
 \multicolumn{1}{c}{}& \multicolumn{1}{c}{}&  \multicolumn{9}{c}{Number of time stages} \\ \hline
Method &Statistic &  1 & 3 & 6 & 9 & 12 & 15 & 18 & 21 & 24    \\ \hline
BGGN & $10$th prct. & 3.5  & 9.8 & 1.7  & 18.8 & 8.5 & 4.6 & 24.8 & 23.5 & 6.6   \\ 
   & Mean &17.3  &  21.0 & 20.8  & 42.7   & 47.9  &  43.7 & 99.2 & 129.3  & 191.2   \\ 
   & $90$th prct. & 39.2  & 38.3  & 36.8   & 70.6   & 100.5 & 94.6  & 154.9  & 225.4  & 762.7   \\ \hline
PDR3 & $10$th prct. & 0 & 0 & - & - & - & - & - & - & -  \\ 
   & Mean & 0  & -0.1 & -  & -  & - & - & - & - & -  \\ 
   & $90$th prct. & 0 & -0.2 & - & - & - & - & - & - & -   \\ 
\hline\hline
\end{tabular}
\caption{Relative gaps (in percent) between the alternative approximation schemes and PLDR}
\label{tab:inventory_value}
\end{table}

\begin{table}[tbp]
\centering
\begin{tabular}{c|ccccccccc}  
 \multicolumn{1}{c}{}& \multicolumn{9}{c}{Number of time stages} \\ \hline
Method & 1  & 3 & 6 & 9 & 12 & 15 & 18 & 21 & 24    \\ \hline
PLDR & 0.02 & 0.29 & 2.31 & 9.66  & 34.60 & 99.29 & 248.40 & 541.75 & 1050.91 \\ 
BGGN & 0.01 & 0.04 &  0.33 & 1.23  & 4.76 & 14.48 & 36.85 & 94.49 & 191.30   \\  
PDR3 & 0.13 & 28.17 & - & - & - & - & - & - & -  \\ 
\hline\hline
\end{tabular}
\caption{The average computation times (in seconds) of the different approximation schemes}
\label{tab:inventory_time}
\end{table}

We also assess the quality of the decision rules obtained from the different approximation methods by evaluating their true worst-case profits. Since the inventory control problem has non-fixed recourse, the worst-case profit of a fixed decision does not necessarily correspond to an extreme point scenario of the uncertainty set $\Xi$. As a result, we design a simulation procedure to estimate the worst-case profit. Specifically, we randomly generate 10,000 samples from the uncertainty set, where each sample point corresponds to a trajectory realization of the demands and prices over the $T$ periods. The approximate worst-case profit is then given by the sample point that generates the smallest profit. 
Table~\ref{tab:inventory_sim} reports the statistics of relative gaps between the worst-case-scenario profit of our method and BGGN and PDR3.  We find that the proposed PLDR scheme still provides substantial average improvements over BGGN method.

\begin{table}[tbp]
\centering
\begin{tabular}{c|c|ccccccccc}  
 \multicolumn{1}{c}{}& \multicolumn{1}{c}{}&  \multicolumn{9}{c}{Number of time stages} \\ \hline
Method &Statistic &  1 & 3 & 6 & 9 & 12 & 15 & 18 & 21 & 24    \\ \hline
BGGN & $10$th prct. & 1.7  & 6.9 & 1.5  & 12.8 & 6.9 & 3.7 & 22.3 & 19.4 & 6.1   \\ 
   & Mean & 15.3  &  19.6 & 17.9  & 35.9   & 42.8  &  41.6 & 79.2 & 119.9  & 181.7   \\ 
   & $90$th prct. & 34.7  & 36.1  & 34.5   & 61.7   & 89.1 & 90.3  & 132.8  & 187.8  & 692.4   \\ \hline
PDR3 & $10$th prct. & 0 & 0 & - & - & - & - & - & - & -  \\ 
   & Mean & 0  & -0.1 & -  & -  & - & - & - & - & -  \\ 
   & $90$th prct. & 0 & -0.1 & - & - & - & - & - & - & -   \\ 
\hline\hline
\end{tabular}
\caption{Relative gaps (in percent) of the worst-case-scenario profits between the alternative approximation schemes and PLDR}
\label{tab:inventory_sim}
\end{table}

\subsection{Index tracking}

For the last example, we study a dynamic index tracking problem, which aims at matching the performance of a 
stock index as closely as possible with a portfolio of other financial instruments over a finite discrete planning horizon
$T$. A stochastic programming version of the problem is described in~\cite{Hanasusanto.Kuhn.2013}. 
To this end, we consider five stock indices, where  the first  four constitute the tracking instruments while 
the last one corresponds to the target index. 
Let $\bm \xi \in \RR^5_+$ be the vector of  total  returns (price relatives) of these indices from time stage $t-1$ to 
time stage $t$.  Here,  $\xi_{t,1}$,  $\xi_{t,2}$,  $\xi_{t,3}$,  and  $\xi_{t,4}$ are  returns of the four tracking 
instruments, while $\xi_{t,5}$ is return of the target index at time stage $t$. The robust dynamic index tracking problem is
stated as follows:
\begin{equation} \label{equ:index}
\begin{array}{ll}
\min &  \max\limits_{\bm \xi \in \Xi} \, \sum\limits_{t=1}^T |\xi_{t,5} - s_t(\bm \xi^t)| \\
\st    & \bm x_0 \geq 0, \ \bm{e}^\top \bm x_0 \leq 1, \ s_1(\bm \xi^1) = \bm \xi_1^\top \bm x_0  \\
& s_t(\bm \xi^t) = \bm \xi_t^\top \bm x_{t-1}(\bm \xi^{t-1}) \ \ \ \forall \, t \in [T]\backslash \{1\}   \\
& \bm{e}^\top \bm x_{t}(\bm \xi^{t}) \leq s_{t}(\bm \xi^{t}), \ \bm x_t(\bm \xi^t) \geq 0  \ \ \ \forall \, t \in [T],
\end{array}
\end{equation} 
The decision variable $s_t(\bm \xi^t) \in \RR_+$ determines the value of the  tracking portfolio 
at time stage $t$.  Here, we aim to rebalance the portfolio allocation vector $\bm x(\bm \xi^t) \in \RR^4$ 
of the four tracking instruments such that $s_t(\bm \xi^t)$ is as close to $\xi_{t,5}$ as possible throughout the planning 
time horizon. 
 The  uncertainty set $\Xi$ in \eqref{equ:index} is specified through a factor model as follows:
\[
\Xi := \left\{ \bm \xi := (\bm \xi_1, \ldots, \bm \xi_T) :
 \begin{array}{l} 
\bm \xi_t = \bm f + \bm F \bm \zeta_t , \ \bm \zeta_t \in \RR^3 \ \ \ \forall \, t \in [T] \\ 
\|\bm \zeta_t\|_\infty \leq 1, \ \|\bm \zeta_t\|_1 \leq \rho \ \ \ \forall \, t \in [T]\end{array} \right\} 
\]
The associated cone $\K$ is accordingly written as
\[
\K := \left\{ \bm \xi, \tau) \in \RR^{4T} \times \RR_+ :
 \begin{array}{l} 
\bm \xi_t = \bm f \tau + \bm F \bm \zeta_t , \ \bm \zeta_t \in \RR^3 \ \ \ \forall \, t \in [T] \\ 
\|\bm \zeta_t\|_\infty \leq \tau, \ \|\bm \zeta_t\|_1 \leq \rho\tau \ \ \ \forall \, t \in [T]\end{array} \right\}. 
\]

Since the objective function of (\ref{equ:index}) is not linear, we introduce auxiliary 
variables $w_t(\cdot)$ to linearize each absolute term. This yields the multi-stage robust linear optimization problem
\begin{equation} \label{equ:index_lp}
\begin{array}{ll}
\min &  \max\limits_{\bm \xi \in \Xi} \, \sum\limits_{t=1}^T w_t(\bm \xi^t) \\
\st    & \bm x_0 \geq 0, \ \bm{e}^\top \bm x_0 \leq 1, \ s_1(\bm \xi^1) = \bm \xi_1^\top \bm x_0  \\
        & s_t(\bm \xi^t) = \bm \xi_t^\top \bm x_{t-1}(\bm \xi^{t-1}) \ \ \ \forall \, t \in [T]\backslash \{1\}   \\
        & \bm{e}^\top \bm x_{t}(\bm \xi^{t}) \leq s_{t}(\bm \xi^{t}), \ \bm x_t(\bm \xi^t) \geq 0  \ \ \ \forall \, t \in [T] \\
       & w_t(\bm \xi^t) \geq \xi_{t,5} - s_t(\bm \xi^t), \ w_t(\bm \xi^t) \geq s_t(\bm \xi^t) -\xi_{t,5}  \ \ \ \forall \, t \in [T]. 
\end{array}
\end{equation} 
As problem (\ref{equ:index_lp}) has non-fixed recourse, we apply linear decision rules to the decision variables 
$\bm x_{t}(\cdot)$, $t\in[T]$, which are multiplied with some uncertain parameters. On the other hand, we may utilize 
quadratic decision rules on 
$s_t(\cdot)$ and $w_t(\cdot)$, $t\in[T]$, as they are not multipled with any uncertain parameters.
With minimum modification, the copositive approach introduced in Section \ref{sec:multi-stage} 
can be applied and, accordingly, we can solve the semidefinite approximation 
which results from replacing the copositive cone $\COP(\K)$ with the inner approximation 
$\IA(\K)$ defined in (\ref{equ:innercone}). We denote our approach by LQDR.
We compare LQDR with the scheme proposed by Ben-Tal et al.~(BGGN)~where 
we replace the cone $\COP(\K)$  with the inner approximation $\AS(\K)$ defined in~(\ref{equ:ascone}), 
and with polynomial decision rule scheme of degree~$3$ (PDR3). 

All experimental results  are averaged over 25 randomly generated instances. 
We utilize the mechanism in~\cite{Hanasusanto.Kuhn.2013} to set up the parameters and to generate the random instances.
For each instance, 
$\bm f$ is set to the vector of all ones, while each entry of  $\bm F$ is sampled uniformly from the interval~$[-1,1]$. 
We further normalize each row of
$\bm F$  such that the sum of the absolute values in each row equals to 1. We test the different schemes on 
problem instances with planning horizons $T = 1$, $3$,  $6$,  $9$, $12$, $15$, and $18$. 
Note that PDR3 can only solve instances up to 
 $T=3$. 
Table~\ref{tab:index_value} reports the statistics of relative gaps between the 
optimal values obtained from LQDR and those  from the  two alternative 
approximation schemes, while Table~\ref{tab:index_time} shows the average computation times for all  three approximation
schemes.  As indicated in Table~\ref{tab:index_value}, the relative gap between LQDR
and BGGN increases  with the planning horizon, where the largest average improvement of $18.2\%$ is observed for $T=18$. 
On the other hand, LQDR generates similar performance to PLDR3 but with significantly less computational effort.

\begin{table}[tbp]
\centering
\begin{tabular}{c|c|ccccccccc}  
\multicolumn{1}{c}{}& \multicolumn{1}{c}{}&  \multicolumn{7}{c}{Number of time stages} \\ \hline
Method &Statistic &  1 & 3 & 6 & 9 & 12 & 15 & 18    \\ \hline
BGGN & $10$th prct. & 0.0  & 1.0 & 1.9  & 2.2 & 4.7 & 1.8 & 2.5    \\ 
   & Mean                 &0.0  &  7.1 & 12.5  & 11.8   & 14.2  &  17.0 & 18.2   \\ 
   & $90$th prct.      & 0.0  & 21.7  & 29.4   & 29.0   & 33.8 & 30.1  & 34.2     \\ \hline
PDR3 & $10$th prct. & 0.0 & 0.0 & - & - & - & - & -   \\ 
                                     & Mean & 0.0  & -0.1 & -  & -  & - & - & -   \\ 
   & $90$th prct.         & 0.0 & -0.4 & - & - & - & - & -    \\ 
\hline\hline
\end{tabular}
\caption{Relative gaps (in percent) between the alternative approximation schemes and LQDR}
\label{tab:index_value}
\end{table}

\begin{table}[tbp]
\centering
\begin{tabular}{c|ccccccccc}  
\multicolumn{1}{c}{}& \multicolumn{7}{c}{Number of time stages} \\ \hline
Method & 1  & 3 & 6 & 9 & 12 & 15 & 18     \\ \hline
LQDR & 0.03 & 0.40 & 5.01& 32.95 & 127.34  & 601.57 & 1703.32  \\ 
BGGN & 0.02 & 0.08 &  0.69 & 5.47 & 24.70 & 75.48 & 226.18   \\  
PDR3 & 0.09 & 8.50   & - & - & - & - & -   \\ 
\hline\hline
\end{tabular}
\caption{The average computation times (in seconds) of the different approximation schemes}
\label{tab:index_time}
\end{table}

We also assess the quality of the decision rules obtained from the different approximation methods by evaluating their worst-case risks. As in the inventory control problem, the index tracking problem has a non-fixed recourse. We adopt the same simulation procedure  by using 10,000 sample trajectories from the uncertainty set to estimate the worst-case risks. 
Table~\ref{tab:index_sim} reports the statistics of relative gaps between the worst-case-scenario risk of our method and BGGN and PDR3.  We find that the proposed PLDR scheme provides substantial average improvements over BGGN method.

\begin{table}[tbp]
\centering
\begin{tabular}{c|c|ccccccccc}  
\multicolumn{1}{c}{}& \multicolumn{1}{c}{}&  \multicolumn{7}{c}{Number of time stages} \\ \hline
Method &Statistic &  1 & 3 & 6 & 9 & 12 & 15 & 18    \\ \hline
BGGN & $10$th prct. & 0.0  & 0.1 & 1.2  & 1.6 & 4.2 & 1.4 & 1.9    \\ 
   & Mean                 &0.0  &  5.2 & 10.4  & 9.7   & 12.1  &  13.7 & 14.1   \\ 
   & $90$th prct.      & 0.0  & 18.5  & 24.2   & 25.2   & 31.0 & 27.4  & 29.2     \\ \hline
PDR3 & $10$th prct. & 0.0 & 0.0 & - & - & - & - & -   \\ 
                                     & Mean & 0.0  & -0.0 & -  & -  & - & - & -   \\ 
   & $90$th prct.         & 0.0 & -0.3 & - & - & - & - & -    \\ 
\hline\hline
\end{tabular}
\caption{Relative gaps (in percent) of the simulated worst-case-scenario risks between the alternative approximation schemes and LQDR}
\label{tab:index_sim}
\end{table}

\section{Concluding remarks}
Generic MSRO problems (with non-fixed recourse) have so far resisted strong decision rule approximations. In this paper, 
we leveraged modern conic programming techniques to derive an exact convex copositive program for  the linear  decision 
rule approximation of these difficult optimization problems. We further  derived an equivalent copositive program for the more 
powerful quadratic decision rule approximation of   instances  with fixed recourse. These  reformulations enabled us to obtain 
a new semidefinite approximation  that is provably tighter than an  existing scheme of similar complexity by Ben-Tal et al.  
The copositive approach further inspired us to develop a new piecewise decision rule   scheme for the generic problems. 
For MSRO problems with non-fixed recourse, we proved that the resulting approximation is tighter than the state-of-the-art 
scheme by Georghiou et al. Extensive numerical results demonstrate that our scheme can  substantially outperform existing 
schemes in terms of optimality, while maintaining  scalability when solving large problem instances. We conclude that our proposed copositive approach provides an excellent balance between optimality and scalability.

We mention two promising directions for further research. First, it would be interesting to derive a copositive programming 
reformulation for the piecewise decision rule scheme where we simultaneously optimize for the best folding directions and 
breakpoints. Second, it is imperative to design
a global solution approach for MSRO problems with non-fixed recourse that leverages the proposed decision rule schemes. 

{
	\paragraph{Acknowledgements.}  Grani A.~Hanasusanto 
	is supported by the National Science Foundation grant no.~1752125.
	The work presented in this paper was carried out while Guanglin Xu 
	was a postdoctoral fellow at the Institute for Mathematics and its Applications during the IMA's annual 
	program on {\em Modeling, Stochastic Control, Optimization, and Related Applications}.
}

\section*{Appendix}

\subsection*{Proof of Lemma \ref{lem:zero}}
\begin{proof}
Fix $(\bm z, \tau) \in \K$. 
For the sake of contradiction, suppose that $\tau =0$ but $\bm z \neq 0$. By Assumption~\ref{ass:bounded}, the set 
$\Uzero$ is nonempty. Choose any $\uu \in \Uzero$, so that $\uu \in \K$ and 
$u_{K+1} = 1$. Then, for any non-negative scalar $\rho \geq 0$, we have $\bm w(\rho) := \uu + \rho (\bm z^\top, 0)^\top \in \K$. Furthermore, $\bm w(\rho) \in \Uzero$ as
$[\bm w(\rho)]_{K+1} = 1$. Since $\rho$ can be arbitrarily large while $\bm z \neq 0$, 
we conclude that $\Uzero$ is unbounded, contradicting the compactness condition of Assumption~\ref{ass:bounded}. Thus, the claim follows. 
\end{proof}

\subsection*{Proof of Lemma \ref{lem:interior}}
\begin{proof}
We prove the statement by showing that the dual problem (\ref{equ:qp-cop}) 
admits a Slater point. To this end, we set $\alpha_i = 0$, $i \in [I]$.
We then seek for a scalar $\lambda$  that ensures 
\begin{equation*}
(\bm z^\top, \tau)  \left(\lambda\mathbf{e}_{K+1}\mathbf{e}_{K+1}^\top - \bm{\widehat C}_0\right) (\bm z^\top, \tau)^\top =\lambda \tau^2 - (\bm z^\top, \tau) \bm{\widehat C}_0 (\bm z^\top, \tau)^\top>0
\end{equation*}
for all non-zero vector $(\bm z, \tau)$ in $\K$. 
By Lemma~\ref{lem:zero}, it suffices to consider the case where $\tau>0$, in which case we may divide the expression by $\tau^2$. We thus require that 
 $\lambda - ((\bm z / \tau)^\top, 1) \bm{\widehat C}_0 ((\bm z / \tau)^\top, 1)^\top$
is strictly positive for all $(\bm z, \tau) \in \K$, $\tau>0$. Since $(\bm z, \tau) \in \K$, we have that $(\bm z/\tau, 1) \in \K$, and, by construction,   $(\bm z/\tau, 1)\in \Uzero$. In this case, the boundedness of $\Uzero$ implies that there exists a constant $\lambda^\star$ such that $\lambda^\star > ((\bm z / \tau)^\top, 1) \bm{\widehat C}_0 ((\bm z / \tau)^\top, 1)^\top$ for all $(\bm z/\tau, 1)\in \Uzero$. The claim thus follows since the point $(\lambda,\bm\alpha)=(\lambda^\star,\bm 0)$ constitutes a Slater point for the problem \eqref{equ:qp-cop}.   
\end{proof}

\subsection*{Proof of Theorem \ref{the:ldr-ref}}
\begin{proof}
Using Lemmas~\ref{lem:burer}
and~\ref{lem:interior}, we can reformulate the maximization problem in the objective function of 
(\ref{equ:ldr}) as a copositive minimization problem. To this end, for any fixed 
decision rule coefficients $\bm Y\in\RR^{N\times (K+1)}$, 
we consider the maximization problem given by
\begin{equation} \label{equ:ldr-obj} 
\sup \limits_{\uu \in \U}  \, (\bm{\widehat D} \uu)^{\top} \bm{Y}\uu.
\end{equation}
By Lemma~\ref{lem:burer}, the problem can be reformulated as a
linear program over the cone of completely positive matrices with respect to $\K$, as follows:
\begin{equation} \label{equ:ldr-epicon-ref-dual}
\begin{array}{ll}
\sup & \dfrac12 \left( \bm{\widehat D}^\top\bm{Y} + \bm{Y}^\top \bm{\widehat D} \right) \bullet \bm{U} \\
\st    &  \mathbf{e}_{K+1}\mathbf{e}_{K+1}^\top \bullet \bm{U} = 1 \\
        & \bm {\widehat C}_i \bullet \bm U = 0 \ \ \ \forall \, i \in [I] \\
        & \bm U \in \SYM^{K+1}, \  \bm{U} \in \CP(\K)
\end{array}
\end{equation}
Letting $\lambda$ and $\bm \alpha$ be the dual variables corresponding to
the constraints $\mathbf{e}_{K+1}\mathbf{e}_{K+1}^\top \bullet \bm{U} = 1$ 
and $\bm {\widehat C}_i \bullet \bm U = 0$, $i \in [I]$, respectively, the dual problem is 
written as:
\begin{equation} \label{equ:ldr-epicon-ref}
\begin{array}{ll}
\inf & \lambda   \\
\st  & \lambda\, \mathbf{e}_{K+1}\mathbf{e}_{K+1}^\top -  \dfrac12 \left( \bm{\widehat D}^\top\bm{Y} + \bm{Y}^\top \bm{\widehat D} \right) + \sum\limits_{i=1}^I \alpha_i \bm{\widehat C}_i \in \COP(\K) \\
& \lambda \in \RR, \ \bm \alpha \in \RR^I.  
\end{array}
\end{equation}
In view of Lemma~\ref{lem:interior}, strong duality holds for the primal and dual pair,
i.e., the optimal value of problem (\ref{equ:ldr-obj})
coincides with that of problem (\ref{equ:ldr-epicon-ref}). Replacing the
maximization problem in (\ref{equ:ldr}) with the minimization problem in \eqref{equ:ldr-epicon-ref} yields the objective function 
and the first constraint in (\ref{equ:cop_ldr}). 

Next, using standard techniques from robust optimization, we
reformulate the semi-infinite constraints in \eqref{equ:ldr} into a finite constraint system. By substituting the definition of problem parameters
$\A(\uu)$, $\BB(\uu)$, and $\hh(\uu)$, and using the definitions in \eqref{equ:ABmat}, 
we can simplify the semi-infinite constraints in \eqref{equ:ldr} to the constraints
\begin{equation*} \label{equ:ldr-sem}
\uu^\top\bm{\widehat \Theta}_j \x + \uu^\top \bm{\widehat \Lambda}_j \bm{Y}\uu \geq \bm{\widehat h}_j^\top \uu \ \ \ \forall \, \uu \in \U \ \  \forall \, j \in [J],
\end{equation*}
where $\bm{\widehat \Theta}_j$ and $\bm{\widehat \Lambda}_j$ are defined as in
(\ref{equ:ABmat}). For any fixed $(\x, \bm{Y}) \in \RR^M \times \RR^{N \times (K+1)}$, we consider 
the $j$-th constraint separately, which can equivalently be stated as 
\begin{equation} \label{equ:icons}
\inf \limits_{\uu \in \U } \, \left(\uu^\top\bm{\widehat \Theta}_j \x + \uu^\top \bm{\widehat \Lambda}_j \bm{Y}\uu -  \bm{\widehat h}_j^\top \uu \right)  \geq 0.
\end{equation}
By Lemma \ref{lem:burer}, the minimization problem on the left-hand side of
(\ref{equ:icons}) can be reformulated as the following linear program over
the cone of completely positive matrices:
\begin{equation} \label{equ:icons-cp}
\begin{array}{ll}
\inf   &\bm{\Omega}_j(\bm x, \, \bm Y)  \bullet \bm{U}_j \\
\st    & \mathbf{e}_{K+1}\mathbf{e}_{K+1}^\top \bullet \bm{U}_j = 1 \\
       & \bm{\widehat C}_i \bullet \bm U_j = 0 \ \ \ \forall \, i \in [I] \\
       & \bm U_j \in \SYM^{K+1}, \ \bm{U}_j \in \CP(\K).
\end{array}
\end{equation} 
Letting $\pi_j \in \RR$ and $\bm \beta_j \in \RR^I$ be the dual
variables corresponding to the constraints $\mathbf{e}_{K+1}\mathbf{e}_{K+1}^\top \bullet \bm{U}_j = 1$
and $\bm{\widehat C}_i \bullet \bm U_j = 0$, $i \in [I]$, respectively, the dual problem is 
given by
\begin{equation} \label{equ:icons-cpp}
\begin{array}{ll}
\sup& \pi_j \\
\st    & \bm{\Omega}_j(\bm x, \, \bm Y) -  \pi_j \, \mathbf{e}_{K+1}\mathbf{e}_{K+1}^\top - \sum\limits_{i=1}^I [\bm \beta_j]_i \, \bm{\widehat C}_i  \in \COP(\K) \\
&\pi_j \in \RR, \ \bm \beta_j \in \RR^I.  
\end{array}
\end{equation}
If the conditions in Assumption \ref{ass:bounded} hold, then by Lemmas~\ref{lem:burer}
and~\ref{lem:interior}, the optimal value of the left-hand side problem in 
(\ref{equ:icons}) coincides with that of problem (\ref{equ:icons-cpp}). 
The emerging 
constraint is satisfied if and only if there exists $\pi_j \geq 0$ and $\bm \beta_j \in \RR^I$ 
such that
\[
\begin{array}{l}
  \bm{\Omega}_j \left( \bm x, \, \bm Y\right)   - \pi_j \, \mathbf{e}_{K+1}\mathbf{e}_{K+1}^\top - \sum\limits_{i=1}^I [\bm \beta_j]_i\, \bm{\widehat C}_i   \in \COP(\K).
\end{array}
\]
Combining the result for all $J$ constraints yields the finite constraint 
system
\[
 \bm{\Omega}_j\left( \bm x, \, \bm Y\right)   - \pi_j \, \mathbf{e}_{K+1}\mathbf{e}_{K+1}^\top - \sum\limits_{i=1}^I  [\bm \beta_j]_i \, \bm{\widehat C}_i   \in \COP(\K), \ \pi_j \geq 0 \  \ \ \forall \, j \in [J],
\]
which completes the proof. 
\end{proof}

\subsection*{Proof of Theorem \ref{thm:qp}}
\begin{proof}
The proof parallels that of Theorem \ref{the:ldr-ref}. Using Lemma \ref{lem:burer}, 
we reformulate the maximization problem $\sup_{\uu \in \U} \sum_{n=1}^N \widehat d_n \uu^\top \bm{Q}_n \uu$ in the objective function of (\ref{equ:qdr-ref}) into a copositive minimization problem given by
\begin{equation*}
\begin{array}{ll}
\inf & \lambda  \\
\st &  \lambda \, \mathbf{e}_{K+1}\mathbf{e}_{K+1}^\top - \sum \limits_{n=1}^N \widehat d_n \bm{Q}_n - \sum\limits_{i=1}^I  \alpha_i \bm{\widehat C}_i  \in \COP(\K) \\
     & \lambda \in \RR, \ \bm \alpha \in \RR^I. 
\end{array}
\end{equation*}
Then, replacing the maximization problem in (\ref{equ:qdr-ref}) with the above minimization problem yields  the objective function and the first constraint in (\ref{equ:cop_qdr}). 

Next, we can reformulate the  constraint 
\[
 \uu^\top\bm{\widehat \Theta}_j \x +  \sum\limits_{n=1}^N  \widehat b_{j n} \uu^\top\bm{Q}_n\uu \geq \bm{\widehat h}_j^\top \uu \ \ \ \forall \, \uu \in \U
\]
corresponding to the $j$-th semi-infinite constraint in \eqref{equ:qdr-ref} into the equivalent constraints
\begin{equation*}
\begin{array}{l}
 \bm{\Gamma}_j\left( \bm x, \bm Q_1, \ldots, \bm Q_N\right) - \pi_j \, \mathbf{e}_{K+1}\mathbf{e}_{K+1}^\top - \sum\limits_{i=1}^I [\bm \beta_j]_i \, \bm{\widehat C}_i   \in \COP(\K),  \  \pi_j \geq 0. 
\end{array}
\end{equation*}
Combining this result for all $J$ semi-infinite constraints yields the  second constraint system
in (\ref{equ:cop_qdr}). This completes the proof.
\end{proof}

\subsection*{Proof of Lemma \ref{lem:linear_soc}}

\begin{proof}
For any $\uu \in \K$, the second-order cone constraint  $\bm{\widehat R}\uu \in \SOC(K_r)$ in the description of $\K$ stipulates that 
\begin{equation}
\label{eq:SOC_in_K}
\mathbf{e}_{K_r}^\top\bm{\widehat R}\uu \geq \sqrt{(\mathbf{e}_1^\top\bm{\widehat R}\uu)^2 + \cdots + (\mathbf{e}_{K_r-1}^\top\bm{\widehat R}\uu)^2}. 
\end{equation}
Squaring both sides of the inequality yields 
\begin{align*}
\uu^\top\bm{\widehat R}^\top \mathbf{e}_{K_r} \mathbf{e}_{K_r}^\top\bm{\widehat R}\uu \geq  \uu^\top\bm{\widehat R}^\top \mathbf{e}_1 \mathbf{e}_1^\top\bm{\widehat R}\uu + \cdots + \uu^\top\bm{\widehat R}^\top \mathbf{e}_{K_r-1} \mathbf{e}_{K_r-1}^\top\bm{\widehat R}\uu & \ \ \ \Longleftrightarrow \\ 
\uu^\top \left( \bm{\widehat R}^\top \mathbf{e}_{K_r} \mathbf{e}_{K_r}^\top\bm{\widehat R}  - \sum\limits_{\ell = 1}^{K_r-1}\bm{\widehat R}^\top \mathbf{e}_\ell \mathbf{e}_\ell^\top\bm{\widehat R}\right)\uu \geq 0 & \ \ \ \Longleftrightarrow \\
\uu^\top \bm{\widehat S}\uu \geq 0 
\end{align*}
Thus, the claim follows. 
\end{proof}

\subsection*{Proof of Proposition \ref{prop:ia-cop}}
\begin{proof}
For any  $\bm V \in \IA(\K)$, we need to show that $ \uu^\top \bm V \uu \geq 0$ for all  $\uu \in \K$.  To this end,  fix any $\bm V \in \IA(\K)$ and $\uu \in \K$. 
By construction, we have 
\begin{align*}
&\ \uu^\top \left( \bm W+\tau \bm{\widehat S} + \bm{\widehat P}^\top\bm{\Sigma}\bm{\widehat P} + \bm{\Psi}\right)\uu \\
= &\  \uu^\top \bm W \uu + \tau \uu^\top \bm{\widehat S}\uu + \uu^\top \bm{\widehat P}^\top\bm{\Sigma}\bm{\widehat P} \uu + \uu^\top \left(\frac12\bm{\widehat P}^\top \bm{\Phi} \bm{\widehat R} +  \frac12\bm{\widehat R}^\top \bm{\Phi}^\top\bm{\widehat P}\right)\uu \\
= & \   \uu^\top \bm W \uu + \tau \uu^\top \bm{\widehat S}\uu + \uu^\top \bm{\widehat P}^\top\bm{\Sigma}\bm{\widehat P} \uu + \uu^\top\bm{\widehat P}^\top \bm{\Phi}\bm{\widehat R}\uu.
\end{align*}
We next analyze each of the four summands separately:
\begin{enumerate}
	\item  Since $\bm W \succeq \bm 0$, we have $\uu^\top\bm W\uu \geq 0$. 
	\item Since $\tau\geq 0$ and by Lemma~\ref{lem:linear_soc}, we have $\tau \uu^\top \bm{\widehat S}\uu \geq 0$. 
	\item Since $\bm{\widehat P} \uu \geq \bm 0$ and $\bm \Sigma \geq \bm 0$, we have  
	$(\bm{\widehat P}\uu)^\top\bm \Sigma (\bm{\widehat P} \uu) \geq 0$.
	\item Since $\bm{\widehat R}\uu$ and the vectors $\textup{Rows}(\bm \Phi)$ belong to $\SOC(K_r)$, we have $\bm{\Phi}\bm{\widehat R}\uu \geq \bm 0$ (as a second-order cone is self-dual). This further implies that  $\uu^\top\bm{\widehat P}^\top \bm{\Phi}\bm{\widehat R}\uu 
	= (\bm{\widehat P}\uu)^\top(\bm{\Phi}\bm{\widehat R}\uu) \geq  0$
	as $\bm{\widehat P}\uu \geq \bm 0$. 
\end{enumerate}
This completes the proof. 
\end{proof}
\subsection*{Proof of Proposition \ref{prop:our_vs_Georghiou}}
\begin{proof}
   For a clear proof, we begin with the case when the piecewise linear lifting is only applied  to the first coordinate axis $u_1$, where the breakpoints are given by
	$h_{1}=\underline u_1<h_{2}<\ldots<h_{L}<\overline u_1$. In this case the lifted set $\U'$ and the set $\U^*$ respectively simplify to
		\begin{equation}
	\label{eq:expanded_Georghiou__}
	\Ubar = \left\{
	(\bm w, \uu) \in \RR^{L} \times  \mathcal U :
	\begin{array}{ll}
	\bm z \in\RR_+^{L+1}  \\  
	w_{\ell} = z_{\ell}-z_{\ell+1}& \ell \in [L] \\
	z_{1}=u_{1}-\underline u_1,\;z_{L+1}=0 & \\
	z_{\ell} \geq u_1-h_{\ell} ,\   \overline{u}_1 \geq z_{\ell} &  \ell \in [L+1]  \\
	z_{\ell}(z_{\ell} - u_1+h_{\ell})= 0 &   \ell \in [L+1]
	\end{array}
	\right\},
	\end{equation}
	\begin{equation}
	\label{eq:Ustar}
	\U^* := \left\{ (\bm w, \uu)\in \RR^{L} \times \U:
	\begin{array}{l}
	\bm z \in \RR_+^{L+1}, \; \bm p \in \RR_+^{L+1},\; \bm r\in\RR_+^{L},\; U\in\RR_+ \\
	\bm Z \in\SYM^{L+1},\bm W\in\SYM^L, \ \bm Q\in\RR^{L\times (L+1)}  \\
	\bm B\bm z=\bm w,\;z_{1}=u_{1}-\underline u_1,\; z_{L+1}=0\\
	\bm z \geq u_1\mathbf e-\bm h,\   \overline{u}_1\mathbf e \geq\bm z \\
	\diag(\bm Z)- \bm p+\bm h\circ\bm z=\bm 0 \\
	\bm H\begin{pmatrix} 
	\bm W & \bm Q & \bm r &\bm  w \\
	\bm Q^\top & \bm Z & \bm p & \bm z\\
	\bm r^\top & \bm p^\top &  U &  u_1 \\  
	\bm w^\top & \bm z^\top &  u_1 & 1\end{pmatrix}\bm H^\top\geq \bm 0,\;\;
	\begin{pmatrix}
	\bm W & \bm Q & \bm r & \bm w \\
	\bm Q^\top & \bm Z & \bm p & \bm z \\
	\bm r^\top & \bm p^\top &  U &  u_1 \\  
	\bm w^\top& \bm z^\top &  u_1 & 1
	\end{pmatrix} \succeq\bm 0
	\end{array}
	\right\},
	\end{equation}
	where the matrices $\bm B\in\RR^{L\times (L+1)}$ and $\bm H\in\RR^{(4L+2)\times(2L+3)}$ are defined as 
	\begin{equation*}
	\bm B=\begin{pmatrix}
	1 &-1 & 0 &\cdots & 0& 0&0\\
	0 & 1 & -1 & \cdots & 0& 0 &0\\
	\vdots & \ddots & \ddots & \ddots & \ddots& \vdots&\vdots\\
	0 & \cdots & \cdots & \cdots & 1 & -1&0\\
	0 & \cdots & \cdots & \cdots & \cdots & 1 & -1
	\end{pmatrix}\qquad\textup{and}
	\qquad
	\bm H=\begin{pmatrix}
	\bm  0&\mathbb I&\bm 0& \bm 0\\
	\bm 0 &\mathbb I& -\mathbf e & \bm h\\
	-\mathbb I &\bm B& \bm 0 & \bm 0 \\
	\mathbb I  &-\bm B&\bm 0 & \bm 0 \\
	\end{pmatrix},
	\end{equation*}
	respectively. One can verify that a point from $\U'$ is also a member of $\U^*$ since the constraints in the definition of $\U^*$ consist of
	a SDP relxation of those from $\U'$; Ssee e.g.,~\cite{Burer.2012}.  
	Furthermore, the outer approximation $\U^{**}$ simplies to:
	\begin{equation}
	\label{eq:Ustarstar}
	{\U}^{**} = \left\{
	(\bm w, \uu) \in \RR^{L}  \times \U :
	\begin{array}{ll}
	\displaystyle u_1-\underline u_1 = \sum_{\ell\in[L]}w_{\ell}&  \\
	h_{2} -\underline u_1\geq w_{1} &  \\
	(h_{\ell+1}-h_{\ell})w_{\ell-1} \geq(h_{\ell}-h_{\ell-1})w_{\ell} & \forall  \ell \in [L]\setminus \{1\}
	\end{array}
	\right\}.
	\end{equation}

	We now establish that $\U^*\subseteq \U^{**}$. 
	First,  the constraints $\bm B\bm z=\bm w$,  $z_{1}=u_{1}-\underline u_1$, and $z_{L+1}=0$ in $\U^*$ imply that
	\begin{equation*}
	\sum_{\ell\in[L]} w_{\ell}= z_{1}-z_{2}+\left[\sum_{\ell\in\{2,\dots,L\}} z_{\ell}-z_{\ell+1}\right]=u_1-\underline u_1. 
	\end{equation*}
	Next, since $z_{2}\geq u_1-h_{2}$, we have that $w_{1}=z_{1}-z_{2}\leq u_{1}-\underline u_1-u_1+h_{2}=h_{2}-\underline u_1$. 
	Thus, the first two constraints in $\U^{**}$ are implied by $\U^*$. It remains to show that the final system of inequalities 
	in~$\U^{**}$ are also implied by the constraints in  $\U^*$. 
	By expanding the matrix product in the penultimate constraint of $\U^*$, we find that $\bm Q=\bm B\bm Z$, 
	and the following constraints hold:
	\begin{equation*}
	\begin{array}{c}
	-\bm r+\bm w\circ\bm h=-\bm B\bm p+(\bm B\bm z)\circ\bm h,\ \ 
	-\bm p\mathbf e^\top+\bm Z+\bm z\bm h^\top\geq\bm 0,\ \ -\bm r\mathbf e^\top+\bm B\bm Z+\bm w\bm h^\top \geq \bm 0. 
	\end{array}
	\end{equation*}
	Next, we perform the substitutions $\bm p=\diag(\bm Z)+\bm h\circ\bm z$, $\bm w=\bm B\bm z$ and  $\bm Q=\bm B\bm Z$ 
	to all occurrences of $\bm p$, $\bm w$, and $\bm Q$, respectively, in the above constraint system. We then  get 
	$\bm r=\bm B(\diag(\bm Z)+\bm h\circ\bm z)$, and by further substituting this value,  we arrive at the equivalent constraint system 
	\begin{equation*}
	\begin{array}{c}
	-(\diag(\bm Z)+\bm h\circ\bm z)\mathbf e^\top+\bm Z+\bm z\bm h^\top\geq\bm 0,\ -\bm B(\diag(\bm Z)+\bm h\circ\bm z)\mathbf e^\top+\bm B\bm Z+\bm B\bm z\bm h^\top \geq \bm 0. 
	\end{array}
	\end{equation*}
	For $\ell\in\{2,\dots,L\}$, one can show that these constraints further imply the  following system of linear inequalities:
	\begin{equation}
	\label{eq:main_inequalities}
	\begin{array}{rl}
	\mathbf e_{\ell-1}\mathbf e_{\ell+1}^\top\bullet \bm Z-\mathbf e_{\ell+1}\mathbf e_{\ell+1}^\top\bullet \bm Z+ z_{\ell+1}(h_{\ell-1}-h_{\ell+1})\geq 0\;\;\\[1mm]
	\mathbf e_{\ell}\mathbf e_{\ell+1}^\top\bullet\bm Z-\mathbf e_{\ell+1}\mathbf e_{\ell+1}^\top\bullet\bm Z+ z_{\ell+1}(h_{\ell}-h_{\ell+1})\geq 0\;\;\\[1mm]
	\mathbf e_{\ell-1}\mathbf e_{\ell+1}^\top\bullet\bm Z+\mathbf e_{\ell}\mathbf e_{\ell}^\top\bullet\bm Z-\mathbf e_{\ell-1}\mathbf e_{\ell-1}^\top\bullet\bm Z-\mathbf e_{\ell}\mathbf e_{\ell+1}^\top\bullet\bm Z+z_{\ell-1}(h_{\ell+1}-h_{\ell-1})+z_{\ell}(h_{\ell}-h_{\ell+1})\geq 0\;\;\\[1mm]
	\mathbf e_{\ell-1}\mathbf e_{\ell}^\top\bullet\bm Z+\mathbf e_{\ell+1}\mathbf e_{\ell+1}^\top\bullet\bm Z-\mathbf e_{\ell}\mathbf e_{\ell}^\top\bullet\bm Z-\mathbf e_{\ell-1}\mathbf e_{\ell+1}^\top\bullet\bm Z+z_{\ell+1}( h_{\ell+1}-h_{\ell-1})+z_{\ell}(h_{\ell-1}-h_{\ell})\geq 0\;\;\\[1mm]
	\mathbf e_{\ell}\mathbf e_{\ell+1}^\top\bullet\bm Z-\mathbf e_{\ell}\mathbf e_{\ell}^\top\bullet\bm Z+z_{\ell}(h_{\ell+1}-h_{\ell})\geq 0.\\[1mm]
	\end{array}
	\end{equation}
We further  relax the large semidefinite constraint in $\U^{*}$ into $\mathcal O(L)$ semidefinite constraints involving $3\times 3$ matrices, as follows:
	\begin{equation}
	\label{eq:relaxed_k_semidefinite}
	\bm M_{\ell}:=\begin{pmatrix}
	\mathbf e_{\ell-1}\mathbf e_{\ell-1}^\top\bullet \bm Z& \mathbf e_{\ell-1}\mathbf e_{\ell}^\top\bullet \bm Z & \mathbf e_{\ell-1}\mathbf e_{\ell+1}^\top\bullet \bm Z\\
	\mathbf e_{\ell-1}\mathbf e_{\ell}^\top\bullet \bm Z& \mathbf e_{\ell}\mathbf e_{\ell}^\top\bullet \bm Z & \mathbf e_{\ell}\mathbf e_{\ell+1}^\top\bullet \bm Z \\
	\mathbf e_{\ell-1}\mathbf e_{\ell+1}^\top\bullet \bm Z & \mathbf e_{\ell}\mathbf e_{\ell+1}^\top\bullet \bm Z & \mathbf e_{\ell+1}\mathbf e_{\ell+1}^\top\bullet \bm Z \\
	\end{pmatrix}\succeq\bm 0\qquad\forall \ell\in[L+1]. 
	\end{equation} 
	We now show that the relaxations \eqref{eq:main_inequalities} and \eqref{eq:relaxed_k_semidefinite} are sufficient to imply that 
	\begin{equation}
	\label{eq:Georghiou_convex_hull_inequality}
	\begin{array}{l}
	(h_{\ell+1}-h_{\ell})w_{\ell-1} \geq(h_{\ell}-h_{\ell-1})w_{\ell}
	\Longleftrightarrow (h_{\ell+1}-h_{\ell})z_{\ell-1}+(h_{\ell}-h_{\ell-1})z_{\ell+1} \geq(h_{\ell+1}-h_{\ell-1})w_{\ell},
	\end{array}
	\end{equation} 
	where the equivalence follows from the substitutions $w_{\ell-1}=z_{\ell-1}-z_{\ell}$ and $w_{\ell}=z_{\ell}-z_{\ell+1}$.  
	In order to arrive the desired implication, we require that the optimal value of the following optimization problem  is greater than or equal to $0$:
	\begin{equation} 
	\label{eq:optimization_verify_inequality}
	\begin{array}{cl}
	\inf & (h_{\ell+1}-h_{\ell})z_{\ell-1}+(h_{\ell}-h_{\ell-1})z_{\ell+1} -(h_{\ell+1}-h_{\ell-1})w_{\ell}\\
	\st &\bm M_{\ell}, z_{\ell-1},z_{\ell},\text { and }z_{\ell+1} \text{ satisfy }  \eqref{eq:main_inequalities} \text { and } \eqref{eq:relaxed_k_semidefinite}. 
	\end{array}
	\end{equation}
	By weak duality, the optimal value of this problem is lower bounded by the maximization problem
	\begin{equation*} 
	\begin{array}{cl}
	\sup &0\\
	\st 
	&  (h_{\ell+1}-h_{\ell-1})c=(h_{\ell+1}-h_{\ell})\\
	& (h_{\ell}-h_{\ell-1}) d + (h_{\ell+1}-h_{\ell})c= (h_{\ell+1}-h_{\ell-1})+ (h_{\ell+1}-h_{\ell})e\\
	& (h_{\ell}-h_{\ell-1}) + (h_{\ell+1}-h_{\ell-1})a+(h_{\ell+1}-h_{\ell})b=(h_{\ell+1}-h_{\ell-1})d\\
	&\displaystyle \begin{pmatrix}
	c & -\frac{d}{2}&\frac{d-a-c}{2}\\
	-\frac{d}{2} & {-c+d+e}& \frac{-b+c-e}{2}\\
	\frac{d-a-c}{2} &  \frac{-b+c-e}{2} & a+b-d
	\end{pmatrix}\succeq 0\\
	& (a,b,c,d,e)\in\RR_+^5. 
	\end{array}
	\end{equation*}
	One can verify that the solution $(a,b,c,d,e)\in\RR_+^5$ satisfying 
	$a=c=\tfrac{h_{\ell+1}-h_{\ell}}{h_{\ell+1}-h_{\ell-1}}$, $b=\tfrac{h_{\ell+1}-h_{\ell-1}}{h_{\ell+1}-h_{\ell}}$, $d=2$, and $e=\tfrac{(h_{\ell}-h_{\ell-1})^2}{(h_{\ell+1}-h_{\ell-1})(h_{\ell+1}-h_{\ell})}$
	is feasible to the dual problem. Thus, the optimal value of the primal problem \eqref{eq:optimization_verify_inequality} is bounded below by $0$, which  verifies  that the constraints \eqref{eq:main_inequalities} and \eqref{eq:relaxed_k_semidefinite}  imply \eqref{eq:Georghiou_convex_hull_inequality}. 
	
	The above analysis also holds to the general case when the piecewise linear lifting is applied  to all coordinate axis $\uu$.
	 In summary, we have shown that the containment $\U^*\subseteq \U^{**}$ holds. 
	
\end{proof}

\subsection*{Proof of Theorem \ref{thm:our_vs_Georghiou}}
\begin{proof}
For a clear proof, we begin with the case when the piecewise linear lifting is only applied  to the first coordinate axis $u_1$, where the breakpoints are given by
	$h_{1}=\underline u_1<h_{2}<\ldots<h_{L}<\overline u_1$. In this case the lifted set $\U'$ simplifies to the one shown in~\eqref{eq:expanded_Georghiou__}.
	
		
	We  apply linear decision rules on the lifted uncertain parameters, which gives rise to the following semi-infinite linear program:
	\begin{equation} \label{equ:pldr_}
	\begin{array}{rcll} 
	&\inf & \cc^{\top} \x + \sup \limits_{(\bm w, \uu) \in \Ubar} \bm{\widehat d}^{\top}\bm Y (\bm w^\top, \bm u^\top)^\top  \\
	&\st  &  \A'(\vv)\x +  \bm{\widehat B}\bm Y   (\bm w^\top, \bm u^\top)^\top \geq \hh'(\vv) \ \ \ \forall \, \bm v:=(\bm w,\uu) \in \U'  \\
	&    & \x \in \X, \, \bm Y\in \RR^{N\times (L+K+1)}.
	\end{array}
	\end{equation}
	Consider the worst-case maximization problem in the objective function of \eqref{equ:pldr_}. For a fixed decision rule coefficient matrix $\bm Y$, let us denote its optimal value by $v(\bm Y)$. That is, 
	\begin{equation} \label{equ:ldr-obj-fixed} 
	v(\bm Y)=\sup \limits_{(\bm w, \uu) \in \Ubar}  \, \bm{\widehat d}^\top\bm{Y} (\bm w^\top, \bm u^\top)^\top.
	\end{equation}
	Replacing the set $\Ubar$ with the outer approximation given by \eqref{eq:Ustarstar}
	 yields the upper bound $v^{**}(\bm Y)\geq v(\bm Y)$. 
	A tractable finite reformulation  can then be derived by virtue of standard dualization technique in robust optimization.  
	
	Alternatively, by applying Proposition \ref{prop:EDR} to the lifted  set $\U'$ in \eqref{eq:expanded_Georghiou__} and using 
	Lemma~\ref{lem:burer}, we arrive at the   equivalent completely positive program
	\begin{equation} \label{equ:pldr-cp}
	\begin{array}{rrl}
	v(\bm Y)=&\sup &   \bm{\widehat d}^\top\bm Y(\bm w^\top, \bm u^\top)^\top\\
	&\st &\mathbf{e}_{K+1}\mathbf{e}_{K+1}^\top \bullet \bm U' = 1\\
	&& \mathbf e_{\ell}\mathbf e_{\ell}^\top\bullet\bm Z'- \mathbf e_{\ell}\mathbf e_{1}^\top\bm \bullet \bm P'+h_{\ell}\mathbf e_{\ell}\mathbf e_{K+1}^\top\bullet \bm P'=0\ \ \ \forall \,  \ell\in[L+1] \\
	&& \begin{pmatrix}
	\bm W' & \bm Q' & \bm R' \\
	(\bm Q')^\top & \bm Z' & \bm P' \\
	(\bm R')^\top & (\bm P')^\top & \bm U'
	\end{pmatrix}  \in \CP(\K'),\ \bm u=\bm U'\mathbf e_{K+1},\ \bm w=\bm R'\mathbf e_{K+1}\\
	&&\bm U' \in \SYM^{K+1},\;\bm Z'\in\SYM^{L+1},\;\bm W'\in\SYM^{L}\\
	&&\bm P'\in\mathbb R^{(L+1)\times (K+1)},\;\bm R'\in\mathbb R^{ L\times (K+1)},\;\bm Q'\in\mathbb R^{L\times L+1},
	\end{array}
	\end{equation}
	where the cone $\mathcal K'$ is defined as
	\begin{equation*}
	\K' := \left\{  
	\begin{array}{l}
	(\bm w, \bm z, \uu) \in \RR^{L}  \times \RR_+^{L+1} \times \K:
	\begin{array}{ll}
	\quad	w_{\ell} = z_{\ell}-z_{\ell+1}& \ell \in [L]\\
	\quad z_{1}=u_{1}-\underline u_1,\; z_{L+1}=0\\
	\quad z_{\ell} \geq u_1-h_{\ell}u_{K+1},\; \overline{u}_1u_{K+1} \geq z_{\ell}  &  \ell \in [L+1] 
	\end{array}
	\end{array}
	\right\}.
	\end{equation*}
	An upper bound to  $v(\bm Y)$ is then obtained by replacing the completely positive 
	cone $\CP(\K')$ in \eqref{equ:pldr-cp} with a valid semidefinite-representable outer approximation. 
	To this end, we further loosen the relaxation by  considering only those constraints that are independent across  dimensions. We then obtain an outer approximation defined by ~\eqref{eq:Ustar} to 
	the feasible set of  decision variables $\bm u$ 
	and $\bm w$ in~\eqref{equ:pldr-cp}.
	Using $\U^*$ to replace $\U$ in \eqref{equ:ldr-obj-fixed}, we arrive at another upper bound 
	${ v}^*(\bm Y)\geq v(\bm Y)$. As the resulting maximization problem admits a Slater point, 
	a tractable finite reformulation can then be obtained by applying standard conic duality.  
    By Proposition~\ref{prop:our_vs_Georghiou}, we establish that $\U^*\subseteq \U^{**}$, which, in turn, establises that $v^{*}(\bm Y)\leq v^{**}(\bm Y)$. 
	
	The above analysis also holds to the general case when the piecewise linear lifting is applied  to all coordinate axis $\uu$. Thus, the claim follows. 
\end{proof}

\end{onehalfspace}

\bibliographystyle{plain}

\bibliography{qdr}

\appendix

\end{document}